\documentclass[a4paper, 12pt]{amsart}
\usepackage{geometry}

\usepackage{youngtab}           
\usepackage{hyperref, latexsym, amssymb, amsfonts, amsxtra, amsmath, mathrsfs, amsthm  }
\usepackage{tikz}
\usepackage{fullpage}

\theoremstyle{plain}
\newtheorem{theorem}{Theorem}[section]
\newtheorem{proposition}[theorem]{Proposition}
\newtheorem{lemma}[theorem]{Lemma}
\newtheorem{corollary}[theorem]{Corollary}

\newtheorem{claim}[theorem]{Claim}

\theoremstyle{definition}
\newtheorem{defn}[theorem]{Definition}
\newtheorem{example}[theorem]{Example}
\newtheorem{remark}[theorem]{Remark}
\numberwithin{equation}{section}

\DeclareMathOperator{\F}{\mathbb{F}} 
 \DeclareMathOperator{\Z}{\mathbb{Z}}
\DeclareMathOperator{\C}{\mathbb{C}}  \DeclareMathOperator{\Q}{\mathbb{Q}}
 
 \DeclareMathOperator{\STD}{Std}

\input xy
\xyoption{all}

\newcommand{\nc}{\newcommand}
\nc{\dten}{10} \nc{\deleven}{11} \nc{\dtwelve}{12}
\nc{\dthirteen}{13} \nc{\dfourteen}{14} \nc{\dfifteen}{15}
\nc{\dsixteen}{16}

\newcommand{\arc}{\ar@{-}@/_/}
\newcommand{\tra}{\ar@{-}}
\newcommand{\dta}{\ar@{.}}
\newcommand{\ltra}{\ar@{<-}}
\newcommand{\rtra}{\ar@{->}}
\newcommand{\larc}{\ar@{<-}@/_/}
\newcommand{\rarc}{\ar@{<-}@/_/}

\begin{document}
\title[$q$-Brauer ]{A cellular basis of the $q$-Brauer algebra related with Murphy bases of Hecke algebras}

\author[D.T. Nguyen]{Dung Tien Nguyen}

\address{Nguyen Tien Dung }

\address{ Department of Education,  Vinh University, Leduan 182, Vinh city, Viet Nam}

\email{dungnt@vinhuni.edu.vn}

\maketitle
\bigskip

\begin{quote}
{\footnotesize {\bf Abstract.}  A new basis of the $q$-Brauer algebra is introduced, which is a lift of Murphy bases of Hecke algebras of symmetric groups. This basis is a cellular basis in the sense of Graham and Lehrer.
Subsequently, using combinatorial language we prove that the non-isomorphic simple $q$-Brauer modules are indexed by the $e(q^2)$-restricted partitions of $n-2k$ where $k$ is an integer, $0 \le k \le [n/2]$. 
When the $q$-Brauer algebra has low-dimension a criterion of semisimplicity is given, which is used to show that the $q$-Brauer algebra is in general not isomorphic to the BMW-algebra.}
\end{quote}

\maketitle
\tableofcontents

\section{Introduction}
In the classical Schur-Weyl duality the actions of the general linear group $GL(N)$  and the symmetric group $S_n$ on the tensor power spaces $(\C^{N})^{\otimes{n}}$
are centralizers of each other. In 1937, Richard Brauer showed that when replacing $GL(N)$ by the orthogonal subgroup $O(N)$ or the symplectic subgroup $Sp(N)$ the corresponding centralizer is a larger algebra  containing the symmetric group, called the Brauer algebra $D_n(N)$. In the quantum case, there is an analogue of these dualities in which:
$GL(N)$ and $S_n$ are substituted by the quantized enveloping algebra $U_q(\mathfrak{gl}_N)$ and the Hecke algebra of the symmetric group $H_n(q)$ respectively (see \cite{Jim}); $O(N)$ (resp. $Sp(N)$) and $D_n(N)$ are substituted by the quantized enveloping algebras $U_q(\mathfrak{o}_N)$ (resp. $U_q(\mathfrak{sp}_N)$ and the BMW-algebra $\mathfrak{B}_n$, a $q$-deformation of the Brauer algebra, with appropriate choices of parameters respectively (see \cite{LR}, or \cite{CP} Section 10.2).

Recently, another $q$-deformation of the Brauer algebra has been introduced by Wenzl \cite{W2} via generators and relations who called it the $q$-Brauer algebra.
This algebra contains the Hecke algebra of the symmetric group as a subalgebra and, over the field $\Q(r,q)$, is semisimple and isomorphic to the Brauer algebra.
Some applications of this algebra were found by Wenzl in \cite{W3} and \cite{W4}. 
In \cite{N} the generic $q$-Brauer algebra is shown to be cellular without giving a cellular basis. Relating with structural properties, it is exhibited as an iterated inflation of Hecke algebras of symmetric groups, and then the complete set of its non-isomorphic simple modules is classified, which bases on this structure. These results in \cite{N} motivate for our work in the present article. In particular, the subjects of this note are following three questions.

{\bf{Question 1.}} How to give a cellular basis of the $q$-Brauer algebra?

{\bf{Question 2.}} Is there a combinatorial and direct proof for parametrization of simple modules of the $q$-Brauer algebra shown in \cite{N}?

{\bf{Question 3.}} In general, does there exist an algebra isomorphism between the $q$-Brauer algebra and the BMW-algebra? 

For Question 1 we construct a new basis of the $q$-Brauer algebra, which is derived from the one introduced in \cite{N} (Theorem 3.13). This new basis is a lift of Murphy bases of Hecke algebras of symmetric groups and exists for every version (one or two parameters) of the $q$-Brauer algebra over a field of any characteristic. The main result stated in Theorem \ref{Mthm} is that the $q$-Brauer algebra over a commutative ring has the basis consisting of elements that are indexed by two pairs, in each pair the first entry is a standard tableaux and the second one is a certain partial Brauer diagram. Then this basis is checked straightforward to be cellular in the sense of Graham and Lehrer \cite{GL}. More precisely, this basis enables us to answer the other two questions. In \cite{N} Dung showed that the simple $q$-Brauer modules up to isomorphism are indexed by the $e(q^2)$-restricted partitions of $n-2k$ where $k$ is an integer, $0 \le k \le [n/2]$. 
 The proof, however, needs to use the structure "iterated inflation" of the $q$-Brauer algebra which is complicated. 
By detail calculations on an explicit basis of cell modules we give a simple answer for the Question 2 in Theorem \ref{irremodule} which does not relate to the structure of the $q$-Brauer algebra. Finally, Question 3 is fully answered by applying the general theory of cellular algebra on the constructed basis of the $q$-Brauer algebra. We give  
a criterion for semisimplicity of the $q$-Brauer algebra, $Br_n(r^2, q^2)$, in the case $n \in \{2, 3\}$ in Propositions \ref{criterion1}, \ref{criterion2} and some explicit calculations in Examples \ref{ex10}, \ref{ex11}, \ref{ex12}. These results imply a negative answer for Question 3. The statement is that:

\begin{claim}  \label{claim} In general, there does not exist an algebra isomorphism between the $q$~-~Brauer algebra $Br_n(r^2, q^2)$ (resp. $Br_n(r, q)$)  and the BMW- algebra $\mathscr{B}_n$.
\end{claim}
 

\medskip
{\bf{Acknowledgments}}

I would like to thank my supervisor, Professor Steffen K\"onig, for constant support and valuable advice during this work.
The research work is financially supported by the Project MOET-322 of the Training and Education Ministry of Vietnam
and (partly) by the DFG Priority Program SPP-1489. I would also like to express my gratitude for this.

\section{Notation and preliminaries}\label{prelsec}
This section recalls the concepts tableaux and Young subgroup and collects basic and necessary facts of the representation theory of the Hecke algebra of the symmetric group. We introduce these with a slight difference in which the usual symmetric group and its deformation, the Hecke algebra, are replaced by isomorphic ones, written in a different way. In particular, we need to use background on a subgroup of the symmetric group and the representation theory of its corresponding Hecke algebra. However, the usual results in the literature hold true for this restriction (see \cite{DJ1}, \cite{Math} or \cite{M2}).  
\subsection{Combinatorics} \label{com}
For a positive integer $n$, denote $S_{n}$ the symmetric group acting on $\{1,\dots,n\}$ on the right.
For $i$ an integer, $1 \le i<n$, let $s_i$ denote the basic transposition, that is, a permutation of the form $(i,i+1)$. Then $S_{n}$ is generated by generators $s_{1},s_{2},\dots,s_{n-1}$, which satisfy the relations
\begin{align*}
&s_i^2=\bf{1}&&\text{for $1\le i<n$;}\\
&s_is_{i+1}s_i=s_{i+1}s_is_{i+1}&&\text{for $1\le i\le n-2$;}\\
&s_is_j=s_js_i&&\text{for $1 < |i-j|$.}
\end{align*}
Let $k$ be an integer, $0\le k\le [n/2]$. Denote $S_{2k+1, n}$ to be the subgroup of $S_n$ generated by generators $s_{2k+1}, s_{2k+2},\cdots, s_{n-1}$. This subgroup is isomorphic to the symmetric group $S_{n-2k}$
 
For $w \in S_n$, if $\omega = s_{i_1}s_{i_2}\cdots s_{i_m}$ and $m$ is minimal with this property
then $\ell(\omega)= m$, and we call $s_{i_1}s_{i_2}\cdots s_{i_m}$ \textit{a reduced expression} for $\omega$.

Let $k$ be an integer, $0\le k\le [n/2]$. If $n-2k>0$,
a \emph{partition} $\lambda$ of $n-2k$ is a sequence
$\lambda =(\lambda_1, \lambda_2, \cdots)$ of non-negative integers such that $\lambda_i \geq \lambda_{i+1}$ for all $i \geq 1$ and $| \lambda | = \sum_{i=1}\lambda_i=n-2k$.
The non-negative integers $\lambda_i$, for $i \geq 1$, are the parts of $\lambda$;
if $\lambda_i = 0$ for $i > m$ we identify $\lambda$ with $(\lambda_1, \lambda_2, \cdots, \lambda_m)$ and denote $\lambda\vdash n-2k$. If $n-2k=0$, write
$\lambda=\varnothing$ for the empty partition.
The \emph{Young diagram} of a partition $\lambda$ is the subset
\begin{align*}
[\lambda]=\{(i,j)\,:\,\text{$1\le i$ and $1 \le j \le \lambda_i$}\,\}\subseteq \mathbb{N}\times\mathbb{N}, 
\end{align*}
where each pair $(i,j)$ of $[\lambda]$ is called a \emph{node} of $\lambda$. The diagram $[\lambda]$ is represented as an array of boxes with $\lambda_i$
boxes on the $i$--th row. For example, if $\lambda=(3, 2, 1)$ then
$[\lambda]=\text{\tiny\Yvcentermath1$\yng(3,2,1)$}$\,. 
Let $k$ be an integer, $0\le k\le [n/2]$, and $\lambda \vdash n-2k$. A $\lambda$--tableau labeled by
$\{2k+1,2k+2,\dots,n\}$ is a bijection $\mathfrak{t}$ from the nodes
of the diagram $[\lambda]$ to the integers $\{2k+1,2k+2,\dots,n\}$.
A given $\lambda$--tableau
$\mathfrak{t}:[\lambda]\to\{2k+1,2k+2,\dots,n\}$ can be visualized
by labeling the nodes of the diagram $[\lambda]$ with the integers
$2k+1,2k+2,\dots,n$. For instance, if $n=10$, $k=2$ and
$\lambda=(3,2,1)$,
\begin{align}\label{tabex0.0}
\mathfrak{t}=\text{\tiny\Yvcentermath1$\young(859,7\dten,6)$}
\end{align}
represents a $\lambda$--tableau. A $\lambda$--tableau $\mathfrak{t}$
labeled by $\{2k+1,2k+2,\dots,n\}$ is said to be \emph{standard} if
the entries in $\mathfrak{t}$ increase from left to right in each row and from top to bottom in each column.
Let $\mathfrak{t}^\lambda$ denote the $\lambda$-tableau
in which the integers $2k+1,2k+2,\dots,n$ are
entered in increasing order from left to right along the rows of
$[\lambda]$.The tableau $\mathfrak{t}^\lambda$ is referred to as the
\emph{superstandard tableau}. For instance, let $n=10$, $k=2$ and
$\lambda=(3,2,1)$,
\begin{align*} 
\mathfrak{t}^\lambda=\text{\tiny\Yvcentermath1$\young(567,89,\dten)$}\,.
\end{align*}
For $\lambda \vdash n-2k$, denote $\STD(\lambda)$ the set of standard $\lambda$--tableaux labeled by the integers
$\{2k+1,2k+2,\dots,n\}$.  

For $\lambda$ and $\mu$ arbitrary partitions, the partition $\lambda$ is called to \emph{dominate} the partition $\mu$, 
write $\lambda\unrhd\mu$, if either
 (1) $|\mu|>|\lambda|$ or
(2) $|\mu|=|\lambda|$ and
$\sum_{i=1}^m\lambda_i\ge\sum_{i=1}^m\mu_i$ for all $m>0$.
We will write $\lambda\rhd\mu$ to mean that $\lambda\unrhd\mu$ and
$\lambda\ne\mu$.
The symmetric group $S_{2k+1, n}$ acts on the set of
$\lambda$--tableaux on the right in the usual manner, by permuting
the integer labels of the nodes of $[\lambda]$. For example,
\begin{align}\label{tabex0}
\text{\Yvcentermath1$\young(567,89,\dten)$}\,(5, 6, 10, 9, 7, 8)\,
=\text{\Yvcentermath1$\young(859,7\dten,6)$} \,.
\end{align}

Let $\lambda$ be a partition of $n-2k$, define \emph{Young subgroup} $S_\lambda$ to be the
row stabilizer of $\mathfrak{t}^\lambda$ in $S_{2k+1,n}$.
For instance, when $n=10$, $k=2$ and $\lambda=(3,2,1)$, then a direct calculation yields $S_\lambda=\langle
s_5,s_6,s_8\rangle$. To each $\lambda$--tableau $\mathfrak{t}$,
associate a unique permutation $d(\mathfrak{t})\in~S_{2k+1,n}$ by
the condition $\mathfrak{t}=\mathfrak{t}^\lambda d(\mathfrak{t})$.
Using the tableau $\mathfrak{t}$ in~\eqref{tabex0.0} above it deduces that $d(\mathfrak{t})~=~(5, 6, 10, 9, 7, 8)$
by~\eqref{tabex0}.

\subsection{The Hecke algebra of the symmetric group}\label{ihsec}
Let $R$ denote an integral domain
and $q$ is an invertible element in $R$. The Hecke algebra of
the symmetric group $S_{2k+1, n}$ is the unital associative $R$--algebra
$H_{2k+1, n}(q^2)$ with generators $g_{2k+1},g_{2k+2},\dots, g_{n-1},$ which
satisfy the defining relations
\begin{align*}
&g^2_i= (q^2-1)g_i+q^2&&\text{for $2k+1\le i<n$;}\\
&g_ig_{i+1}g_i=g_{i+1}g_ig_{i+1}&&\text{for $2k+1\le i<n-1$;}\\
&g_ig_j=g_jg_i&&\text{for $2\le|i-j|$.}
\end{align*}
Note that $H_{2k+1, n}(q^2)$ is isomorphic to the usual Hecke algebra $H_{n-2k}(q^2)$ in the literature.
If $w\in S_{2k+1, n}$ and $s_{i_1}s_{i_2}\cdots s_{i_m}$ is a
reduced expression of $w$, then
$g_w=g_{i_1}g_{i_2}\cdots g_{i_m}$
is a well defined element of $H_{2k+1, n}(q^2)$ and the set
$\{g_w\,:\,w\in S_{2k+1, n}\}$ is a basis the Hecke algebra
$H_{2k+1, n}(q^2)$. 
From now on, we abbreviate $H_{2k+1, n}$ replacing $H_{2k+1, n}(q^2)$.
Let $*$ denote the algebra involution of $H_{2k+1, n}$ determined by $(g_w)^* = g_{w^{-1}}$ for $\omega~\in~S_{2k+1, n}$. Write $g^*_w = (g_w)^*$.

In the following we collect some facts from the
representation theory of the Hecke algebra of the symmetric
group that need for our subsequent work in Sections 3 and 4; details can be found in~\cite{Math}
or~\cite{M2}. If $\mu$ is a partition of $n-2k$, define the element
\begin{align} \label{ct16}
c_\mu=\sum_{\sigma\in S_\mu} g_\sigma.
\end{align}
Denote $\check{\mathscr{H}}^\lambda_{2k+1, n}$ to be the $R$-module in $H_{2k+1, n}$ with basis
\begin{align} \label{idealh}
\big\{c_\mathfrak{st}=g_{d(\mathfrak{s})}^*c_\mu
g_{d(\mathfrak{t})}:\text{$\mathfrak{s},\mathfrak{t}\in\STD(\mu)$,
where $\mu\rhd\lambda$ }\big\}.
\end{align}
The next statement is due to Murphy in~\cite{M2}.

\begin{theorem}\label{Hthm}
The Hecke algebra $H_{2k+1, n}$ is free as an
$R$--module with basis
\begin{align}\label{ct7'}
\mathscr{M}=\left\{c_\mathfrak{st}= g_{d(\mathfrak{s})}^*c_\lambda
g_{d(\mathfrak{t})}\,\bigg|\,
\begin{matrix}
\text{ for $\mathfrak{s},\mathfrak{t}\in\STD(\lambda)$ and }\\
\text{$\lambda$ a partition of $n-2k$}
\end{matrix}
\right\}.
\end{align}
Moreover, the following statements hold.
\begin{enumerate}
\item The $R$--linear involution $*$ satisfies $(c_{\mathfrak{st}})^* = c_\mathfrak{ts}$ for all $\mathfrak{s},\mathfrak{t}\in\STD(\lambda)$.
\item Suppose that $h\in H_{2k+1, n}$,
and $\mathfrak{s}$ is a standard $\lambda$--tableau. Then
there exist $a_\mathfrak{t}\in R$, for
$\mathfrak{t}\in\STD(\lambda)$, such that for all
$\mathfrak{s}\in\STD(\lambda)$,
\begin{align}\label{ct6}
c_{\mathfrak{sv}}h\equiv
\sum_{\mathfrak{t}\in\STD(\lambda)}a_\mathfrak{t}
c_{\mathfrak{st}} \mod \check{\mathscr{H}}^\lambda_{2k+1, n}.
\end{align}
\end{enumerate}
\end{theorem}

The basis $\mathscr{M}$ is cellular in the sense
of~\cite{GL}. If $\lambda$ is a partition of $n-2k$, the
{\it{cell (or Specht)}} module $S^\lambda$ for $H_{2k+1, n}$ is
the $R$--module freely generated by
\begin{align}\label{ct7}
\{c_{\mathfrak{s}}=c_\lambda g_{d(\mathfrak{s})} +\check{\mathscr{H}}^\lambda_{2k+1, n}\,:\,\mathfrak{s}\in
\STD(\lambda)\},
\end{align}
and with the right $H_{2k+1, n}$--action
\begin{align} \label{ct8}
c_{\mathfrak{s}}h=
\sum_{\mathfrak{t}\in\STD(\lambda)}
a_{\mathfrak{t}}c_{\mathfrak{t}},&&\text{for
$h\in H_{2k+1, n}$,}
\end{align}
where the coefficients $a_\mathfrak{t}\in R$, for
$\mathfrak{t}\in\STD(\lambda)$, are determined by the
expression~\eqref{ct6}. The basis $\mathscr{M}$ is called {\it{Murphy basis}} for $H_{2k+1, n}$ and the basis~\eqref{ct7} is referred to
as the Murphy basis for $S^\lambda$. 
Notice that the $H_{2k+1, n}$--module $S^\lambda$ is dual to Specht module in~\cite{DJ1}.

Applying the general theory of cellular algebra, the bilinear form on $S^\lambda$ is the
unique symmetric $R$-bilinear map from $S^\lambda \times S^\lambda$ to $R$ such that
\begin{align} \label{bilinearform1}
\langle  c_{\mathfrak{s}},\ c_{\mathfrak{t}} \rangle  c_{\lambda}\equiv c_{\mathfrak{s}} c^*_{\mathfrak{t}} \text{ mod } \check{\mathscr{H}}^\lambda_{2k+1, n}
\end{align}
for all $\mathfrak{s}, \ \mathfrak{t} \in Std(\lambda)$.
Then, $rad\ S^\lambda=\{ x\in S^\lambda |\ \langle x, y\rangle=0 \text{ for all } y\in S^\lambda\}$ is a \break $H_{2k+1, n}$-submodule of $S^{\lambda}$. For each partition $\lambda$ of $n-2k$, denote $D^\lambda=S^\lambda/rad\ S^\lambda$ a right $H_{2k+1, n}$-module. 

Let $e(q^{2})$ be the least positive integer $m$ such that $[m]_{q^{2}}=1 + q^{2} + q^{4}...+ q^{2(m-1)} = 0$ if that exists, and let $e(q^{2}) = \infty$ otherwise. Recall that a partition 
$\lambda =( \lambda_{1},\ \lambda_{2},\ ..., \lambda_{f})$ of $n-2k$ is $e(q^{2})-restricted$ if $\lambda_i - \lambda_{i+1} < e(q^{2})$ for all $i \geq 1$.

For partitions $\lambda, \mu$ of $n-2k$ and $D^\mu\neq 0$, let $d_{\lambda \mu} = [S^\lambda\ :\ D^\mu]$ be the composition multiplicity of $D^\mu$ in $S^\mu$. The following classification of the simple $H_{2k+1, n}$--modules is given by Dipper and James (see \cite{DJ1}, Theorem~7.6 or \cite{Math}, Theorem~3.43). 

\begin{theorem} \label{simplemd1}
Suppose that $R$ is a field.
\begin{enumerate}
\item \{ $D^\mu$ |\ $\mu$ an $e(q^2)$-restricted partition of $n-2k$\} is a complete set of non-isomorphic simple $H_{2k+1, n}$--modules.
\item Suppose that $\mu$ is an $e(q^2)$-restricted partition of $n-2k$ and that $\lambda$ is a partition of $n-2k$. Then $d_{\mu \mu} = 1$ and $d_{\lambda \mu}  \neq 0$ only if $\lambda \unrhd\mu$.
\end{enumerate}
\end{theorem}

\begin{corollary}(\cite{Math}, Corollary 3.44) \label{cor2}
Suppose that $R$ is a field. Then the following statements are equivalent.
\begin{enumerate}
\item $H_{2k+1, n}$ is (split) semisimple;
\item $S^{\lambda} = D^{\lambda}$ for all partitions $\lambda$ of $n-2k$;
\item $e(q^2) > n-2k$.
\end{enumerate}
\end{corollary}

\subsection{The Brauer algebra}
Brauer algebras were introduced first by Richard Brauer \cite{Br} in order to study the ${n}$-th tensor power of the defining representation of the orthogonal and symplectic groups. Afterwards, they were studied in more detail by various mathematicians. We refer the reader to work of Hanlon and Wales (\cite{HW1, HW2}), Doran, Wales and Hanlon \cite{DWH}, Graham and Lehrer \cite{GL} or K\"onig and Xi~\cite{KX}, Wenzl~\cite{W1} for more information.

The Brauer algebra is defined over the ring $\mathbb{Z}[x]$ by a basis given by diagrams with $2n$ vertices, arranged in two rows, and $n$ edges, where each vertex belongs to exactly one edge.
The edges which connect two vertices on the same row are called \textit{horizontal edges}.
The other ones are called \textit{vertical edges}. We denote by $D_{n}(x)$ Brauer algebra.
The vertices of diagrams are numbered $1$ to $n$ from left to right in both the top and the bottom.
The multiplication of two basis diagrams $d_{1}$ and $d_{2}$ is a concatenation in the following way:
We put diagram $d_{1}$ on top of $d_{2}$ such that all vertices in the bottom row of $d_{1}$ coincide with all upper vertices of $d_{2}$.
Now draw an edge from vertex $i$ in the bottom row of $d_{1}$ to vertex $i$ in top row of $d_{2}$ for all $i$.
The resulting diagram consists of parts that start and finish in top row of $d_{1}$ and bottom row of $d_{2}$ respectively, as well as some cycles that use only vertices in the middle two rows.
Let $\gamma(d_{1},\ d_{2})$ denote the number of these internal cycles. The product $d_{1} \cdot d_{2}$ in $D_{n}(x)$ is then defined to be this resulting diagram without internal cycles,
multiplied by $x$ taken to the power $\gamma(d_{1},\ d_{2})$. Here $x$ is a variable.

\begin{example} Let us consider in $D_{7}(x)$ the product of $d_{1}$ and $d_{2}$
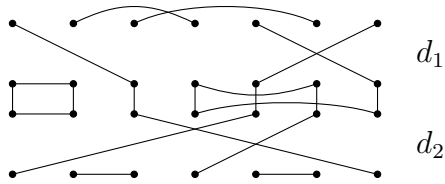
\begin{align*}
\begin{matrix}
\begin{tikzpicture}
\draw (-9.5,0) -- (-7.9,-0.8);
\draw (-8.7,0) .. controls (-8.1,+0.3) and (-7.7,+0.3) .. (-7.1,0);
\draw (-7.9,0) .. controls (-7.3,+0.3) and (-6.1,+0.3) .. (-5.5,0);
\draw (-6.3,0) -- (-4.7,-0.8);
\draw (-4.7,0) -- (-6.3,-0.8);
\draw (-8.7,-0.8) -- (-9.5,-0.8);
\draw (-7.1,-0.8) .. controls (-6.5,-1) and (-6.1,-1) .. (-5.5,-0.8);
\filldraw [black] (-8.7,0) circle (1.2pt);
\filldraw [black] (-8.7,-0.8) circle (1.2pt);
\filldraw [black] (-7.9,0) circle (1.2pt);
\filldraw [black] (-7.9,-0.8) circle (1.2pt);
\filldraw [black] (-7.1,0) circle (1.2pt);
\filldraw [black] (-7.1,-0.8) circle (1.2pt);
\filldraw [black] (-6.3,0) circle (1.2pt);
\filldraw [black] (-6.3,-0.8) circle (1.2pt);
\filldraw [black] (-5.5,0) circle (1.2pt);
\filldraw [black] (-5.5,-0.8) circle (1.2pt);
\filldraw [black] (-4.7,0) circle (1.2pt);
\filldraw [black] (-4.7,-0.8) circle (1.2pt);
\filldraw [black] (-9.5,0) circle (1.2pt);
\filldraw [black] (-9.5,-0.8) circle (1.2pt);
\draw (-9.5,-0.8) -- (-9.5,-1.2);
\draw (-8.7,-0.8) -- (-8.7,-1.2);
\draw (-7.9,-0.8) -- (-7.9,-1.2);
\draw (-7.1,-0.8) -- (-7.1,-1.2);
\draw (-6.3,-0.8) -- (-6.3,-1.2);
\draw (-5.5,-0.8) -- (-5.5,-1.2);
\draw (-4.7,-0.8) -- (-4.7,-1.2);
\node at (-4,-0.4) {$\displaystyle  d_1 \color{black}$};
\node at (-4,-1.6) {$\displaystyle  d_2 \color{black}$};
\draw (-9.5,-1.2) -- (-8.7,-1.2);
\draw (-7.9,-1.2) -- (-4.7,-2);
\draw (-7.1,-1.2) .. controls (-6.5,-1) and (-5.3,-1) .. (-4.7,-1.2);
\draw (-6.3,-1.2) -- (-9.5,-2);
\draw (-5.5,-1.2) -- (-7.1,-2);
\draw (-8.7,-2) -- (-7.9,-2);
\draw (-6.3,-2) -- (-5.5,-2);
\filldraw [black] (-9.5,-1.2) circle (1.2pt);
\filldraw [black] (-9.5,-2) circle (1.2pt);
\filldraw [black] (-8.7,-1.2) circle (1.2pt);
\filldraw [black] (-8.7,-2) circle (1.2pt);
\filldraw [black] (-7.9,-1.2) circle (1.2pt);
\filldraw [black] (-7.9,-2) circle (1.2pt);
\filldraw [black] (-7.1,-1.2) circle (1.2pt);
\filldraw [black] (-7.1,-2) circle (1.2pt);
\filldraw [black] (-6.3,-1.2) circle (1.2pt);
\filldraw [black] (-6.3,-2) circle (1.2pt);
\filldraw [black] (-5.5,-1.2) circle (1.2pt);
\filldraw [black] (-5.5,-2) circle (1.2pt);
\filldraw [black] (-4.7,-1.2) circle (1.2pt);
\filldraw [black] (-4.7,-2) circle (1.2pt);
\end{tikzpicture}
\end{matrix}\ 
\end{align*}

\bigskip
and the resulting diagram is

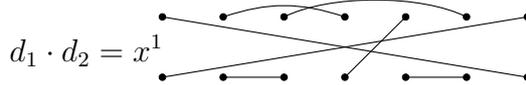
\begin{align*}
\begin{matrix}
\begin{tikzpicture}
\node at (-10.5,-0.45) {$\displaystyle d_1\cdot d_2 = x^1 \color{black}$};
\draw (-9.5,0) -- (-4.7,-0.8);
\draw (-8.7,0) .. controls (-8.1,+0.2) and (-7.7,+0.2) .. (-7.1,0);
\draw (-7.9,0) .. controls (-7.3,+0.3) and (-6.1,+0.3) .. (-5.5,0);
\draw (-6.3,0) -- (-7.1,-0.8);
\draw (-4.7,0) -- (-9.5,-0.8);
\draw (-7.9,-0.8) -- (-8.7,-0.8);
\draw (-6.3,-0.8) -- (-5.5,-0.8);
\filldraw [black] (-8.7,0) circle (1.2pt);
\filldraw [black] (-8.7,-0.8) circle (1.2pt);
\filldraw [black] (-7.9,0) circle (1.2pt);
\filldraw [black] (-7.9,-0.8) circle (1.2pt);
\filldraw [black] (-7.1,0) circle (1.2pt);
\filldraw [black] (-7.1,-0.8) circle (1.2pt);
\filldraw [black] (-6.3,0) circle (1.2pt);
\filldraw [black] (-6.3,-0.8) circle (1.2pt);
\filldraw [black] (-5.5,0) circle (1.2pt);
\filldraw [black] (-5.5,-0.8) circle (1.2pt);
\filldraw [black] (-4.7,0) circle (1.2pt);
\filldraw [black] (-4.7,-0.8) circle (1.2pt);
\filldraw [black] (-9.5,0) circle (1.2pt);
\filldraw [black] (-9.5,-0.8) circle (1.2pt);
\end{tikzpicture}
\end{matrix}\ 
\end{align*}
\end{example}

In (\cite{Br}, Section 5) R. Brauer points out that each basis diagram on $D_{n}(x)$ which has exactly $2k$ horizontal edges
can be obtained in the form $\omega_{1}e_{(k)}\omega_{2}$ where $\omega_{1}$ and $\omega_{2}$ are permutations in $S_{n}$,
and $e_{(k)}$ is the following diagram:

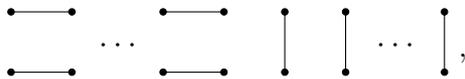
\begin{align*}
\begin{matrix}
\begin{tikzpicture}
\node at (-9,-0.45) {$\displaystyle \cdots \color{black}$};
\node at (-5.35,-0.45) {$\displaystyle \cdots \color{black}$};
\draw (-9.6,0) -- (-10.4,0);
\draw (-9.6,-0.8) -- (-10.4,-0.8);
\draw (-7.6,-0.8) -- (-8.4,-0.8);
\draw (-7.6,0) -- (-8.4,0);
\draw (-6.8,0) -- (-6.8,-0.8);
\draw (-6,0) -- (-6,-0.8);
\draw (-4.7,0) -- (-4.7,-0.8);
\filldraw [black] (-10.4,0) circle (1.2pt);
\filldraw [black] (-10.4,-0.8) circle (1.2pt);
\filldraw [black] (-9.6,0) circle (1.2pt);
\filldraw [black] (-9.6,-0.8) circle (1.2pt);
\filldraw [black] (-8.4,0) circle (1.2pt);
\filldraw [black] (-8.4,-0.8) circle (1.2pt);
\filldraw [black] (-7.6,0) circle (1.2pt);
\filldraw [black] (-7.6,-0.8) circle (1.2pt);
\filldraw [black] (-6.8,0) circle (1.2pt);
\filldraw [black] (-6.8,-0.8) circle (1.2pt);
\filldraw [black] (-6.0,0) circle (1.2pt);
\filldraw [black] (-6.0,-0.8) circle (1.2pt);
\filldraw [black] (-4.7,0) circle (1.2pt);
\filldraw [black] (-4.7,-0.8) circle (1.2pt);
\end{tikzpicture}
\end{matrix}\ ,
\end{align*}

where each row has exactly $k$ horizontal edges.
\subsubsection{Length function for Brauer algebra $D_{n}(N)$} \label{length}
Generalizing the length of elements in reflection groups, Wenzl \cite{W2} defined a length function for a basis diagram of $D_{n}(N)$ as follows.

For a basis diagram $d$ $\in D_{n}(N)$ with exactly $2k$ horizontal edges, the definition of the length $\ell{(d)}$ is given by
$$\ell{(d)}=min\{\ell{(\omega_{1})} +\ell{(\omega_{2})} | \quad \omega_{1}e_{(k)}\omega_{2}=d, \ \omega_{1}, \omega_{2} \in {S_{n}} \}.$$
Recall that here we see a permutation $\omega$ of a symmetric group as a diagram of the Brauer algebra with no horizontal edge. The product $\omega_1 \omega_2$ is a concatenation of two diagrams $\omega_1$ and $\omega_2$.
 
As indicated in \cite{W2}, a permutation $\omega \in S_{n}$ can be written uniquely in the form \break $\omega~=~t_{1} \dots t_{n-2}t_{n-1}$,
where $t_{j} = 1$ or $t_{j} =s_js_{j -1}s_{j - 2}\dots s_{i_j} =: s_{j,i_{j}}$ with $1 \leq{i_{j}} \leq j<{n}$.
Denote $B_k$ the set of all elements of the form $t_{2}t_{4}\dots t_{2k-2}t_{2k}t_{2k+1}\dots t_{n-2}t_{n-1}$.

For $k$ an integer, $0\le k\le [n/2]$, let $\mathscr{D}_{k, n}$ be the set of all diagrams $d$ in which: A diagram has exactly $k$ horizontal edges on each row,
its top row is like a row of the diagram $e_{(k)}$, and there is no crossing between any two vertical edges. Set
\begin{align} \label{bkn}
B_{k, n}= \{ \omega \in B_k \ |
\ell(d)= \ell(\omega) \text{ with } d = e_{(k)} \omega \in \mathscr{D}_{k, n} \}.
\end{align}
This definition is going to be used in the following section on the $q$-Brauer algebra. For more detail we refer the reader to Section 3.3\cite{N}.

\section{A cellular basis of the $q$-Brauer algebra}\label{qBr}
\subsection{The $q$-Brauer algebra}From now on, we abbreviate $H_n$ replacing $\mathscr{H}_n(q^2)$.
The generic $q$-Brauer algebra, which contains the Hecke algebra of the symmetric group $H_n$ as a subalgebra, is defined below.

\begin{defn} \label{DefofBr}  Let $r$ and $q$ be invertible elements over the ring $\Z[q^{\pm{1}}, r^{\pm{1}}, (\dfrac{r-r^{-1}}{q-q^{-1}})^{\pm{1}}]$.
Moreover, if $q=1$ then assume that $r=q^N$ with $N \in \Z \setminus \{0\}$.
The $q$-Brauer algebra $Br_{n}(r^2, q^{2})$ over $\Z[q^{\pm{1}}, r^{\pm{1}}, (\dfrac{r-r^{-1}}{q-q^{-1}})^{\pm{1}}]$ is the algebra defined via generators $g_{1}$, $g_{2}$, $g_{3}$, ..., $g_{n-1}$ and $e$ and relations
\begin{enumerate}
\item[(H)] The elements $g_{1}$, $g_{2}$, $g_{3}$, ..., $g_{n-1}$  satisfy the relations of the Hecke algebra $H_{n}$;
\item[$(E_{1})$] $e^{2} = \dfrac{r-r^{-1}}{q-q^{-1}}e$;
\item[$(E_{2})$]  $eg_{i} = g_{i}e$ for $i > 2,$ $eg_{1} = g_{1}e = q^{2}e$, $eg_{2}e = rqe$ and $eg^{-1}_{2}e = (rq)^{-1}e$;
\item[$(E_{3})$] $g_{2}g_{3}g^{-1}_{1}g^{-1}_{2}e_{(2)} = e_{(2)}g_{2}g_{3}g^{-1}_{1}g^{-1}_{2},$
where $ e_{(2)} = e(g_{2}g_{3}g^{-1}_{1}g^{-1}_{2})e.$
\end{enumerate}
\end{defn}

Let
\[ g^{+}_{l, m} =
\begin{cases}
&g_{l}g_{l+1}...g_{m} \hspace{1.5cm} \text{ if $ l \le m $}; \\
&g_{l}g_{l-1}...g_{m} \hspace{1.5cm} \text{ if $ l > m $},
\end{cases}
\]
and
\[ g^{-}_{l, m} =
\begin{cases}
&g^{-1}_{l}g^{-1}_{l+1}...g^{-1}_{m} \hspace{1.5cm} \text{ if $l \le m $}; \\
&g^{-1}_{l}g^{-1}_{l-1}...g^{-1}_{m} \hspace{1.5cm} \text{ if $ l > m $,}
\end{cases}
\]
for $1 \le l,  m \le n$.

Let $k$ be an integer, $1 \le k \le [n/2]$. The elements $e_{(k)}$ in $Br_{n}(r^2, q^2)$ are defined inductively by $e_{(1)}=e$
and by
\begin{align}\label{e_k}
e_{(k+1)} = eg^{+}_{2, 2k+1}g^{-}_{1, 2k}e_{(k)}.
\end{align}

\begin{remark} \label{rm1} 1. We keep the notation $e_{(k)}$ as used in \cite{W2} and \cite{N}, which describes  both the Brauer diagram $e_{(k)}$ of the Brauer algebra and the element  $e_{(k)}$ (in \eqref{e_k}) of the $q$-Brauer algebra. 

2. The definition of the $q$-Brauer algebra above is a generic version of the one introduced by Wenzl (see \cite{W4}, Definition~2.1). This means when $r=q^{N}$, both definitions are the same. The algebra defined above is also isomorphic to another version of the $q$-Brauer algebra, $Br_n(r, q)$, used by Dung in~\cite{N}.
In fact, $Br_{n}(r^2,q^2)$ can be obtained by in $Br_{n}(r,q)$ we substitute \emph{old} $q$, $r$ and $e$ by $q^{2}$, $r^{2}$  and $(q^{-1}r)e$, respectively.
This implies that the $q$-Brauer $Br_{n}(r^2, q^{2})$ has similar properties as those of $Br_{n}(r, q)$. 

3. To relate the $q$-Brauer algebra to the Brauer algebra over a field of any characteristic, we need another version of the $q$-Brauer algebra. The definition is the following:

Fix $N \in \Z \setminus \{0\}$ and let  $[N] = 1 + q^{2} + \dots + q^{2(N-1)}$, where $q$ is an invertible element in an arbitrary commutative noetherian ring $R$ containing $\Z[q^{\pm{1}}, r^{\pm{1}}, [N]^{\pm{1}}]$. The $q$-Brauer algebra $Br_n(N)$ is an algebra over $R$ defined by generators $g_1,\ g_2,\dots,\ g_{n-1}$ and $e$ and relations $(H),\ (E_3)$ as before and 
\begin{enumerate}
\item[$(E'_{1})$] $e^{2} = [N]e$;
\item[$(E'_{2})$]  $eg_{i} = g_{i}e$ for $i > 2,$ $eg_{1} = g_{1}e = q^{2}e$, $eg_{2}e = q^{N+1}e$ and $eg^{-1}_{2}e = (q)^{-1-N}e$.
\end{enumerate}
It is clear that in the case $q=1$ the $q$-Brauer algebra $Br_n(N)$ coincides with the Brauer algebra $D_n(N)$. Notice that the other versions of the $q$-Brauer algebra recover the Brauer algebra over fields allowing to form the limit $q \rightarrow 1$, such as the field of real or complex numbers (see Remark 3.1(1) in \cite{W2} for more detail).

4. Both $Br_{n}(r^2, q^{2})$ and $Br_n(N)$ have  an $R$-linear involution $*$ defined by
$e^* = e$ and $g^*_i = g_i$ for $1 \le i \le n-1$.  
This involution is the same as the involution of $Br_{n}(r, q)$ shown to exist in Proposition~3.12~\cite{N}, and it is compatible with the involution of the Hecke algebra $H_n$ defined in Section \ref{ihsec}.

In this article, the proofs for both $Br_{n}(r^2, q^{2})$ and $Br_n(N)$ are the same. We will only give them for one version, sometimes without explicitly mentioning the other version.

5. Let $k$ be an integer, $0 \le k \le [n/2]$. By Definition \eqref{e_k} it is straightforward to check that the element $e_{(k)}$ commutes with $g_\omega$ for $\omega \in H_{2k+1, n}$, that is, $e_{(k)}g_\omega =g_\omega e_{(k)}$.

\smallskip
Recall from  \cite{N} Section 4.1 that if $k$ is an integer, $0\le k\le [n/2]$, $J_n(k)$ is the $R$-module generated by the basis elements $g_d$ of the $q$-Brauer algebra, where $d$ is a Brauer diagram whose number of vertical edges are less than or equal $n-2k$. Then $J_n(k)$ is an ideal of the $q$-Brauer algebra and
\begin{align} \label{ct9}
J_n(k) = \sum_{j=k}^{[n/2]} H_{n}e_{(j)}H_{n}.
\end{align}
\end{remark}
In the following we collect some results on $Br_n(r^2, q^2)$ that are similar to those of $Br_n(r, q)$ in Lemmas 3.3, 3.4 and Corollary 3.15\cite{N}. 

\begin{lemma} \label{tc1}
The following statements hold for the $q$-Brauer algebra $Br_{n}(r^2, q^{2})$.
\begin{itemize}
\item[(1)] $g_{2j+1} e_{(k)} = e_{(k)}g_{2j+1} = q^{2}e_{(k)}$, and
$g^{-1}_{2j+1} e_{(k)} = e_{(k)}g^{-1}_{2j+1} = q^{-2}e_{(k)}$ for $0 \leq j < k$;
\item[(2)] $e_{(j)}e_{(k)} = e_{(k)}e_{(j)} = (\dfrac{r-r^{-1}} {q-q^{-1}})^{j}e_{(k)}$  for any $j \leq k$;
\item[(3)] $g^{+}_{2i-1, 2j}e_{(k)} = g^{+}_{2j+1, 2i }e_{(k)}$ and
   $g^{-}_{2i-1, 2j}e_{(k)} = g^{-}_{2j+1, 2i }e_{(k)}$ for $ 1 \leq i \leq j < k$;
\item[(4)] $e_{(k)}g^{+}_{2l, 1} = e_{(k)}g^{+}_{2, 2l+1}$ and
   $e_{(k)}g^{-}_{2l, 1} = e_{(k)}g^{-}_{2, 2l+1}$ for $l < k$;
\item[(5)]$e_{(k)}g_{2j}g_{2j-1} = e_{(k)}g_{2j}g_{2j+1}$
   and $e_{(k)}g^{-1}_{2j}g^{-1}_{2j-1} = e_{(k)}g^{-1}_{2j}g^{-1}_{2j+1}$ for $1 \leq j < k$;
\item[(6)] $e_{(k)}g^{+}_{2j, 2i-1} = e_{(k)}g^{+}_{2i, 2j+1}$ and
   $e_{(k)}g^{-}_{2j, 2i-1} = e_{(k)}g^{-}_{2i, 2j+1}$ for $ 1 \leq i \leq j < k$;
\item[(7)] $e_{(k)}g^{-}_{2k, 2j-1 }g^{+}_{2k+1, 2j}e_{(j)} = (\dfrac{r-r^{-1}} {q-q^{-1}})^{j-1}e_{(k+1)}$ for $1 \leq j < k$;
\item[(8)] $e_{(k)}g_{2j}e_{(j)} = rq(\dfrac{r-r^{-1}} {q-q^{-1}})^{j-1}e_{(k)}$ for $1 \leq j \leq k$;
\item[(9)] $e_{(k)}H_{n}e_{(j)} \subset e_{(k)}H_{2j+1, n} \ + \sum_{m \geq k+1} H_{n} e_{(m)}H_{n},$
where $j \leq k$;
\item[(10)] $e_{(k+1)} = e_{(k)}g^{-}_{2k, 1}g^{+}_{2k+1, 2}e$.
\end{itemize}
\end{lemma}

\begin{lemma}(\cite{N}, Lemma 4.10) \label{lm4}
Let $k,\ l$ be integers, $0< k \leq l \le [n/2]$, and let $u$ be a permutation in $B_{l, n}$ and $\pi$ a permutation in $S_{2l+1, n}$.
Then there exist $a_{(\omega, u)} \in R$, for $v \in B_{k, n}$ and $\omega \in S_{2k+1, n}$, such that
\begin{align*}
e_{(k)}g_{\pi}g_u  = \sum_{\substack{\omega \in S_{2k+1, n} \\ v \in B_{k, n}} } a_{(\omega, v)} e_{(k)} g_\omega g_v.
\end{align*}
\end{lemma}

\begin{theorem}(\cite{N}, Theorem 4.17 ) \label{Cellthm}Suppose that $\Lambda$ is a commutative noetherian ring which contains $R$ as a subring with the same identity. If $q$, $r$ and $\dfrac{r-r^{-1}}{q-q^{-1}}$ (resp. $[N]$) are invertible in $\Lambda$, then the $q$-Brauer algebra $Br_{n}(r^2, q^{2})$ (resp. $Br_n(N)$) over the ring $\Lambda$ is cellular.
\end{theorem}
The following statement gives an explicit basis for the $q$-Brauer algebra $Br_n(r^2,q^2)$. The proof follows from Theorem~3.13, Propositions~3.14 and~4.12 in \cite{N}. Note that the same basis was used to show cellularity of the $q$-Brauer algebra $Br_n(r,q)$ in \cite{N}. That basis, however, was not proved to be a cellular basis.

\begin{theorem}\label{orgthm}
The $q$-Brauer algebra $Br_n(r^2,q^2)$ (resp. $Br_n(N)$) is freely generated as an $R$-module by the basis
\begin{align} \label{rm}
\{\ g^{*}_u e_{(k)}g_\pi g_v \ |\ u, v\in B_{k, n} \text{ and } \pi \in S_{2k+1, n} \text{\ \ for \  } 0 \leq k \leq [n/2]\ \}.
\end{align}
Moreover, the following statements hold.
\begin{enumerate}
\item The involution $*$ satisfies
\begin{align*}
* :g^{*}_ug_\pi e_{(k)}g_v \mapsto g^{*}_{v}g^{*}_{\pi}e_{(k)}g_{u}
\end{align*}
for all  $u, v\in B_{k, n} \text{ and } \pi \in S_{2k+1, n}$.
\item  Suppose that $b\in Br_n(r^2,q^2)$ and let $k$ be an integer, $0\le k\le [n/2]$.\\
If $u, v \in B_{k, n} \text{ and } \pi \in S_{2k+1, n}$, then there exist
$v_1 \in B_{k, n} \text{ and } \pi_1 \in S_{2k+1, n}$ such that
\begin{align}\label{rm:1}
g^{*}_u  e_{(k)}g_\pi g_v b \equiv
\sum_{\substack{\pi_1\in S_{2k+1, n}\\ v_1 \in B_{k, n}}} a_{(\pi_1, v_1)} g^{*}_ue_{(k)}g_{\pi_1}g_{v_1} \text{ mod $J_n(k+1).$ }
\end{align}
\end{enumerate}
\end{theorem}
For $k$ an integer, $0 \le k \le [n/2]$, the $R$-module $J_n(k+1)$ has a basis 
\begin{align} \label{J_{k+1}}
\{\ g^{*}_u e_{(l)}g_\pi g_v \ |\ u, v\in B_{l, n} \text{ and } \pi \in S_{2l+1, n} \text{\ \ for all \  } k < l \leq [n/2]\ \}.
\end{align}

\begin{corollary}\label{lm3}
Let $k$ be an integer, $0< k \le [n/2]$. If $b\in Br_n(r^2,q^2)$, $u\in B_{k, n}$,
then there exist $a_{(\omega, v)} \in R$, for $\omega\in S_{2k+1, n} \text{ and } v \in B_{k, n}$, such that
\begin{align*}
e_{(k)} g_u b\equiv \sum_{\substack{\omega\in S_{2k+1, n}\\ v \in B_{k, n}}} a_{(\omega, v)}e_{(k)}g_{\omega}g_{v} \mod{J_n(k+1)}.
\end{align*}
\end{corollary}

\subsection{Main theorem} \label{Mbasis}
For $k$ an integer, $0 \le k \le [n/2]$, let
\begin{center}
$\Lambda_n := \{ (k, \lambda)\ | \text{ for all}\ 0 \le k \le [n/2], \text{ and } \lambda \text{ is a partition of } n-2k \}.$
\end{center}

For $(k, \mu) \in \Lambda_n$, define the element
\begin{align} \label{ct4}
m_\mu=e_{(k)}c_\mu=c_\mu e_{(k)} \text{\hspace{1cm} where $c_\mu$ is defined in \eqref{ct16}.}
\end{align}
 

\begin{example}
Let $n=10$ and $\mu=(3,2,1)$. The example in \eqref{tabex0} yields the subgroup
$S_\mu=\langle s_5,s_6, s_8\rangle$
and $m_\mu=e_{(2)}\sum_{\sigma\in S_\mu}g_\sigma =e_{(2)}(1+g_5)(1+g_6+g_6g_5)(1+g_8)$.
\end{example}

For $(k, \lambda) \in \Lambda_n$, define
$\mathcal{I}_{n}(k, \lambda)$ to be the set of ordered pairs
\begin{align}\label{index:1}
{\mathcal{I}}_{n}(k,\lambda)= \STD(\lambda) \times  B_{k, n}
= \left\{(\mathfrak{s},u):
\mathfrak{s}\in\STD(\lambda)\text{ and } u\in B_{k, n}
\right\}.
\end{align}
Let $\check{Br}_n^\lambda$ be the $R$-module with spanning set
\begin{align} \label{basisofBr}
\left\{ x^\mu_{(\mathfrak{s}, u)(\mathfrak{t}, v)}:= g^{*}_ug_{d(\mathfrak{s})}^*m_\mu
g_{d(\mathfrak{t})}g_v\,\bigg|\,
\begin{matrix}
\text {$(\mathfrak{s},u),(\mathfrak{t},v)\in \mathcal{I}_{n}(l,\mu)$} \\
\text{$\mu \rhd \lambda$\, for $(l, \mu), (k, \lambda) \in \Lambda_n$}
\end{matrix}
\right\}.
\end{align} 

\begin{lemma} \label{bsofBr}
Suppose that $(k, \lambda) \in \Lambda_n$, then $J_n(k+1)\subseteq\check{Br}_n^\lambda$ and $\check{Br}_n^\lambda$ is an ideal of the $q$-Brauer algebra.  
\end{lemma}
\begin{proof}
By \eqref{J_{k+1}}, every basis element in $J_n(k+1)$ is of the form $g^{*}_u e_{(l)}g_\pi g_v$ where $u, v\in~B_{l, n} $ and $\pi \in S_{2l+1, n}$, $k+1 \leq l \leq [n/2]$. Using Theorem \ref{Hthm}, the element $g_\pi$ can be rewritten $g_\pi~=\sum_{\mathfrak{s}, \mathfrak{t} \in Std(\mu)} g_{d(\mathfrak{s})}^*c_\mu g_{d(\mathfrak{t})}$ with $(l, \mu) \in \Lambda_n$. Since $n-2l < n-2k$, the definition of dominance order in Section \ref{com} implies that $\mu \rhd \lambda$. Thus, Formula \eqref{ct4} and the set \eqref{basisofBr} yield
$$g^{*}_u e_{(l)}g_\pi g_v= g^{*}_u e_{(l)} \big{(}\sum_{\mathfrak{s}, \mathfrak{t} \in Std(\mu)}g_{d(\mathfrak{s})}^*c_\mu g_{d(\mathfrak{t})}\big{)}g_v
 =\sum_{\mathfrak{s}, \mathfrak{t} \in Std(\mu)} g^{*}_u g_{d(\mathfrak{s})}^*m_\mu g_{d(\mathfrak{t})}g_v=\sum_{\mathfrak{s}, \mathfrak{t} \in Std(\mu)}x^\mu_{(\mathfrak{s}, u)(\mathfrak{t}, v)} \in \check{Br}_n^\lambda,$$ 
 that is, $J_n(k+1)\subseteq\check{Br}_n^\lambda$.

To prove the second statement, it is sufficient to show that $x^\mu_{(\mathfrak{s}, u)(\mathfrak{t}, v)}\cdot b \in\check{Br}_n^\lambda$, where $x^\mu_{(\mathfrak{s}, u)(\mathfrak{t}, v)} \in \check{Br}_n^\lambda$ (with $\mu \rhd \lambda$) 
and $b$ is a basis element of the $q$-Brauer algebra $Br_n(r^2, q^2)$. By Corollary \ref{lm3} we obtain
\begin{align*} (e_{(l)}g_v)b \equiv \sum_{\substack{\pi_1\in S_{2l+1, n}\\ v_1 \in B_{l, n}}} a_{(\pi_1, v_1)}e_{(l)}g_{\pi_1}g_{v_1} \mod{J_n(l+1)}.
\end{align*}
Notice that in Formula \eqref{ct7'} of Theorem \ref{Hthm}, $c_{\mathfrak{\boldsymbol{1}}\mathfrak{t}}= \boldsymbol{1}\cdot c_\mu g_{d(\mathfrak{t})}$ where $\boldsymbol{1}$ is the identity of the Hecke algebra.
We have the following calculation:
\begin{align*} 
x^\mu_{(\mathfrak{s}, u)(\mathfrak{t}, v)}\cdot b&= (g^{*}_ug_{d(\mathfrak{s})}^*m_\mu g_{d(\mathfrak{t})}g_v)\cdot b 
= (g^{*}_ug_{d(\mathfrak{s})}^*c_\mu g_{d(\mathfrak{t})})(e_{(l)}g_vb)\\
&\overset{Cor\ref{lm3}}{=}  \big{(}g^{*}_ug_{d(\mathfrak{s})}^*c_\mu g_{d(\mathfrak{t})}\big{)}\big{(}\sum_{\substack{\pi_1\in S_{2l+1, n}\\ v_1 \in B_{l, n}}} a_{(\pi_1, v_1)}e_{(l)}g_{\pi_1}g_{v_1} + J_n(l+1)\big{)}\\
&= \sum_{\substack{\pi_1\in S_{2l+1, n}\\ v_1 \in B_{l, n}}} a_{(\pi_1, v_1)} g^{*}_ug_{d(\mathfrak{s})}^*e_{(l)}(\boldsymbol{1}c_\mu g_{d(\mathfrak{t})}g_{\pi_1})g_{v_1} + J_n(l+1)\\
\end{align*}
\begin{align*}
&\overset{\eqref{ct7'}}{=} \sum_{\substack{\pi_1\in S_{2l+1, n}\\ v_1 \in B_{l, n}}} a_{(\pi_1, v_1)} g^{*}_ug_{d(\mathfrak{s})}^*e_{(l)}(c_{\mathfrak{\boldsymbol{1}}\mathfrak{t}}g_{\pi_1})g_{v_1} + J_n(l+1)\\
&\overset{\eqref{ct6}}{=} \sum_{\substack{\pi_1\in S_{2l+1, n}\\ v_1 \in B_{l, n}}} a_{(\pi_1, v_1)} g^{*}_ug_{d(\mathfrak{s})}^*e_{(l)}(\sum_{\mathfrak{t}_1 \in Std(\mu)} a_{\mathfrak{t}_1} c_{\mathfrak{\boldsymbol{1}}\mathfrak{t}_1} + \check{\mathscr{H}}^\mu_{2l+1, n})g_{v_1} + J_n(l+1)\\
&\overset{\eqref{ct7'}}{=} \sum_{\substack{\pi_1\in S_{2l+1, n}\\ v_1 \in B_{l, n}}} a_{(\pi_1, v_1)} g^{*}_ug_{d(\mathfrak{s})}^*e_{(l)}(\sum_{\mathfrak{t}_1 \in Std(\mu)} a_{\mathfrak{t}_1} c_\mu g_{d(\mathfrak{t}_1)} + \check{\mathscr{H}}^\mu_{2l+1, n})g_{v_1} + J_n(l+1)\\
&\overset{}{=} \sum_{\substack{\pi_1\in S_{2l+1, n}\\ v_1 \in B_{l, n}}}\sum_{\mathfrak{t}_1 \in Std(\mu)} a_{\mathfrak{t}_1} a_{(\pi_1, v_1)} (g^{*}_ug_{d(\mathfrak{s})}^*e_{(l)} c_\mu g_{d(\mathfrak{t}_1)}g_{v_1})\\ 
&+ \sum_{\substack{\pi_1\in S_{2l+1, n}\\ v_1 \in B_{l, n}}} a_{(\pi_1, v_1)} g^{*}_u e_{(l)} (g_{d(\mathfrak{s})}^* \check{\mathscr{H}}^\mu_{2l+1, n}) g_{v_1} + J_n(l+1)\\
&\overset{\eqref{ct4}, \eqref{idealh}}{=} \sum_{\substack{\pi_1\in S_{2l+1, n}\\ v_1 \in B_{l, n}}}\sum_{\mathfrak{t}_1 \in Std(\mu)} a_{\mathfrak{t}_1} a_{(\pi_1, v_1)} (g^{*}_ug_{d(\mathfrak{s})}^*m_\mu g_{d(\mathfrak{t}_1)}g_{v_1})\\ 
&+ \sum_{\substack{\pi_1\in S_{2l+1, n}\\ v_1 \in B_{l, n}}} a_{(\pi_1, v_1)}g^{*}_ue_{(l)}\big{(} \sum_{\substack{\mu_2 \vdash n-2l,\ \mu_2 \rhd \mu \\ \mathfrak{s}_2, \mathfrak{t}_2 \in Std(\mu_2)}}a_{(\mathfrak{s}_2, \mathfrak{t}_2)} g_{d(\mathfrak{s}_2)}^* c_{\mu_2}  g_{d(\mathfrak{t}_2)}\big{)}  g_{v_1} + J_n(l+1)\\
&\overset{\eqref{ct4}}{=} \sum_{\substack{\pi_1\in S_{2l+1, n}\\ v_1 \in B_{l, n}}}\sum_{\mathfrak{t}_1 \in Std(\mu)} a_{\mathfrak{t}_1} a_{(\pi_1, v_1)} (g^{*}_ug_{d(\mathfrak{s})}^*m_\mu g_{d(\mathfrak{t}_1)}g_{v_1})\\ 
&+ \sum_{\substack{\pi_1\in S_{2l+1, n}\\ v_1 \in B_{l, n}}}\sum_{\substack{\mu_2 \vdash n-2l,\ \mu_2 \rhd \mu \\ \mathfrak{s}_2, \mathfrak{t}_2 \in Std(\mu_2)}} a_{(\pi_1, v_1)} a_{(\mathfrak{s}_2, \mathfrak{t}_2)} (g^{*}_ug_{d(\mathfrak{s}_2)}^* m_{\mu_2}  g_{d(\mathfrak{t}_2)}  g_{v_1}) + J_n(l+1)\\
&\overset{}{=} \sum_{\substack{\pi_1\in S_{2l+1, n}\\ v_1 \in B_{l, n}}}\sum_{\mathfrak{t}_1 \in Std(\mu)} a_{\mathfrak{t}_1} a_{(\pi_1, v_1)} (x^\mu_{(\mathfrak{s}, u)(\mathfrak{t}_1, v_1)})\\ 
&+ \sum_{\substack{\pi_1\in S_{2l+1, n}\\ v_1 \in B_{l, n}}}\sum_{\substack{\mu_2 \vdash n-2l,\ \mu_2 \rhd \mu \\ \mathfrak{s}_2, \mathfrak{t}_2 \in Std(\mu_2)}} a_{(\pi_1, v_1)} a_{(\mathfrak{s}_2, \mathfrak{t}_2)} (x^{\mu_2}_{(\mathfrak{s}_2, u)(\mathfrak{t}_2, v_1)}) + J_n(l+1),
\end{align*}
where $a_{(\pi_1, v_1)}, a_{\mathfrak{t}_1}$ and $a_{\mathfrak{s}_2, \mathfrak{t}_2}$ are in $R$, and $\check{\mathscr{H}}^\mu_{2l+1, n}$ is the ideal of $H_{2l+1, n}$ defined in \eqref{idealh}.
Now by the first statement of Lemma \ref{bsofBr}, the assumption $\mu \rhd \lambda$ (hence $l \geq k$) and Definition \eqref{basisofBr}, all elements occuring in the last formula are in $\check{Br}_n^\lambda$, and hence, so is $x^\mu_{(\mathfrak{s}, u)(\mathfrak{t}, v)}\cdot b$.  
\end{proof}

The main result of this section is the following.
\begin{theorem}\label{Mthm}
The $q$-Brauer algebra $Br_n(r^2,q^2)$ (resp. $Br_n(N)$) is freely generated as an $R$-module by
the collection
\begin{align}\label{ct5}
\left\{x^\lambda_{(\mathfrak{s}, u)(\mathfrak{t}, v)} = g^{*}_ug_{d(\mathfrak{s})}^*m_\lambda
g_{d(\mathfrak{t})}g_v\,\bigg|\,
\begin{matrix}
\text {$(\mathfrak{s},u),(\mathfrak{t},v)\in \mathcal{I}_{n}(k,\lambda)$, for $(k, \lambda) \in \Lambda_n$}
\end{matrix}
\right\}.
\end{align}
Moreover, the following statements hold.
\begin{enumerate}
\item The involution $*$ sends
$x^\lambda_{(\mathfrak{s}, u)(\mathfrak{t}, v)}\text{ to }
(x^\lambda_{(\mathfrak{s}, u)(\mathfrak{t}, v)})^*=x^\lambda_{(\mathfrak{t}, v)(\mathfrak{s}, u)}$
for all $(\mathfrak{t},v),(\mathfrak{s},u)~\in~\mathcal{I}_n(k,\lambda)$.
\item  Suppose that $b\in Br_n(r^2,q^2)$, $(k, \lambda) \in \Lambda_n$
and $(\mathfrak{s},u), (\mathfrak{t},v)\in\mathcal{I}_{n}(k,\lambda)$. Then, there exist
$a_{(\mathfrak{t}_2,v_2)}\in R$, for $(\mathfrak{t}_2,v_2)\in\mathcal{I}_{n}(k,\lambda)$, such that
\begin{align}\label{ct10}
x^\lambda_{(\mathfrak{s}, u)(\mathfrak{t}, v)}\cdot b \equiv
\sum_{(\mathfrak{t}_2,v_2)\in\mathcal{I}_{n}(k,\lambda)}a_{(\mathfrak{t}_2,v_2)}
x^\lambda_{(\mathfrak{s}, u)(\mathfrak{t}_2, v_2)}\ \ \mod
\check{Br}^\lambda_n.
\end{align}
\end{enumerate}
\end{theorem}

\begin{proof}
Let $b$ be an arbitrary element in $Br_n(r^2, q^2)$. Then by Theorem~\ref{orgthm}, $b$ can be expressed as an $R$-linear combination
\begin{align}\label{ct2}
b=\sum_{j}a_{j} g^*_{u_j}g_{\pi_j}e_{(k_j)}g_{v_j},
\end{align}
where $(k_j, \lambda) \in \Lambda_n$, $a_j \in R$, $\pi_j \in S_{2k_j+1, n}$, and $\ u_j,\ v_j \in B_{k_j, n}$. Using Theorem~\ref{Hthm} the element $g_{\pi_j}$ has an expression
\begin{align*}
g_{\pi_j}= \sum_p a_{j_p}g^{*}_{d(\mathfrak{s}_{j_p})} c_{\lambda_j} g_{d(\mathfrak{t}_{j_p})}
\end{align*}
with some $\mathfrak{s}_{j_p},\ \mathfrak{t}_{j_p} \in Std(\lambda_j)$, $a_{j_p} \in R$, and $p \in \mathbb{N}$.
Replace $g_{\pi_j}$ in \eqref{ct2} by the last equation,
\begin{align*}
b&=\sum_{j}a_{j} g^*_{u_j} \big{(}\sum_p a_{j_p}g^{*}_{d(\mathfrak{s}_{j_p})} c_{\lambda_j} g_{d(\mathfrak{t}_{j_p})}\big{)}e_{(k_j)}g_{v_j}
=\sum_{j}\sum_p a_{j} a_{j_p} g^*_{u_j} g^{*}_{d(\mathfrak{s}_{j_p})}  e_{(k_j)} c_{\lambda_j} g_{d(\mathfrak{t}_{j_p})}g_{v_j}\\
&\overset{\eqref{ct4}}=\sum_{j}\sum_p a_{j} a_{j_p} g^*_{u_j} g^{*}_{d(\mathfrak{s}_{j_p})} m_{\lambda_j} g_{d(\mathfrak{t}_{j_p})}g_{v_j} =\sum_{j}\sum_p a_{j} a_{j_p}  x^{\lambda_j}_{(\mathfrak{s}_{j_p}, u_j)(\mathfrak{t}_{j_p}, v_j)},
\end{align*}
where $(\mathfrak{s}_{j_p}, u_j), (\mathfrak{t}_{j_p}, v_j) \in \mathcal{I}_{n}(k_j,\lambda_j)$.

Therefore, the set \eqref{ct5} linearly spans $Br_n(r^2, q^2)$. The independence of~\eqref{ct5} follows from the linear independences of \eqref{ct7'} and \eqref{rm} in Theorems \ref{Hthm} and \ref{orgthm}, respectively.

The statement (1) is obtained by combining the involution * on the Hecke algebra $H_n$ and Lemma \ref{tc1}(10) in which $(e_{(k)})^* =e_{(k)}$.

The statement (2) is shown as follows: Let $(k, \lambda), (l, \mu) \in \Lambda_n$.
Applying the analysis of the basis element $b$ above, it suffices to consider the product $x^\lambda_{(\mathfrak{s}, u)(\mathfrak{t}, v)}\cdot b'$, where $(\mathfrak{s}, u)$ and $ (\mathfrak{t}, v)$ are in $\mathcal{I}_{n}(k,\lambda)$ and $b' :=  x^\mu_{(\mathfrak{s}_1, u_1)(\mathfrak{t}_1, v_1)}$ with $(\mathfrak{s_1}, u_1),\ (\mathfrak{t_1}, v_1) \in \mathcal{I}_{n}(l,\mu)$ is a summand of $b$. Subsequently, we consider two cases with respect to partitions $\mu$ and $\lambda$.

The first case is $\mu \vartriangleright \lambda$: Then the definition of the dominance order implies that $k \leq l$. So, by Lemma 3.4~\cite{W2}, we get
$$ e_{(k)}g_v g^*_{u_1} e_{(l)} \in H_{2k+1, n} e_{(l)} + \sum_{m \geq l+1}H_n e_{(m)} H_n.$$
Hence,
\begin{align*}
( c_\lambda g_{d(\mathfrak{t})}) &(e_{(k)}g_v g^*_{u_1} e_{(l)}) (g^*_{d(\mathfrak{s_1})}  c_\mu) \in  c_\lambda g_{d(\mathfrak{t})}H_{2k+1, n} e_{(l)} g^*_{d(\mathfrak{s_1})} c_\mu  + \sum_{m \geq l+1}H_n e_{(m)} H_n\\
&\overset{\eqref{ct4}}\subseteq H_{2k+1, n}g^*_{d(\mathfrak{s_1})} m_\mu  + \sum_{m \geq l+1}H_n e_{(m)} H_n
\overset{k \leq l}\subseteq H_{2k+1, n}g^*_{d(\mathfrak{s_1})} m_\mu  + \sum_{m \geq k+1}H_n e_{(m)} H_n\\
&\overset{\eqref{ct9}}\subseteq H_{2k+1, n}g^*_{d(\mathfrak{s_1})} m_\mu + J_n(k+1) \overset{L\ref{bsofBr}}{\subseteq} \check{Br}^\lambda_n.
\end{align*}
Since $( c_\lambda g_{d(\mathfrak{t})}) (e_{(k)}g_v g^*_{u_1} e_{(l)}) (g^*_{d(\mathfrak{s_1})}  c_\mu)
\overset{\eqref{ct4}}=( m_\lambda g_{d(\mathfrak{t})} g_v) (g^*_{u_1}g^*_{d(\mathfrak{s_1})} m_\mu)$, the preceding calculation yields
\begin{align*}
(x^\lambda_{(\mathfrak{s}, u)(\mathfrak{t}, v)})\cdot(x^\mu_{(\mathfrak{s}_1, u_1)(\mathfrak{t}_1, v_1)})=(g^{*}_ug_{d(\mathfrak{s})}^*m_\lambda g_{d(\mathfrak{t})}g_v)(g^{*}_{u_1}g_{d(\mathfrak{s_1})}^*m_\mu g_{d(\mathfrak{t_1})}g_{v_1}) \in \ g^{*}_ug_{d(\mathfrak{s})}^*\cdot\check{Br}^\lambda_n\cdot g_{d(\mathfrak{t_1})}g_{v_1} \subseteq \check{Br}^\lambda_n.
\end{align*}
Thus, $x^\lambda_{(\mathfrak{s}, u)(\mathfrak{t}, v)})(x^\mu_{(\mathfrak{s}_1, u_1)(\mathfrak{t}_1, v_1)}) \equiv 0 \ (mod \ \ \check{Br}^\lambda_n),$ namely,
$x^\lambda_{(\mathfrak{s}, u)(\mathfrak{t}, v)}\cdot b' \equiv 0 \ (mod \ \ \check{Br}^\lambda_n).$

The second case is $\lambda \trianglerighteq \mu$, that is $l \leq k$: Using Lemma~\ref{tc1}(9), it yields
$$ e_{(k)}g_v g^*_{u_1} e_{(l)} \in e_{(k)}H_{2l+1, n} + \sum_{m \geq k+1}H_n e_{(m)} H_n.$$
Applying Lemma~\ref{tc1}(9) and the same arguments as in the previous case implies that
\begin{align*}
( c_\lambda g_{d(\mathfrak{t})}) &(e_{(k)}g_v g^*_{u_1} e_{(l)}) (g^*_{d(\mathfrak{s_1})}  c_\mu)
=(m_\lambda g_{d(\mathfrak{t})} g_v) (g^*_{u_1}g^*_{d(\mathfrak{s_1})} m_\mu)\\
&\in  c_\lambda g_{d(\mathfrak{t})} e_{(k)} H_{2l+1, n} g^*_{d(\mathfrak{s_1})} c_\mu  + \sum_{m \geq k+1}H_n e_{(m)} H_n\\
&\overset{\eqref{ct9}, \eqref{ct4}}\subseteq  m_\lambda H_{2l+1, n} + J_n(k+1)
\subseteq m_\lambda H_{2l+1, n} + \check{Br}^\lambda_n \ \ \text{(by Lemma \ref{bsofBr})}.
\end{align*}
Hence, in this case
\begin{align*}
x^\lambda_{(\mathfrak{s}, u)(\mathfrak{t}, v)}\cdot b' &= g^{*}_ug_{d(\mathfrak{s})}^*(m_\lambda g_{d(\mathfrak{t})}g_vg^{*}_{u_1}g_{d(\mathfrak{s_1})}^*m_\mu) g_{d(\mathfrak{t_1})}g_{v_1}\\
&\in g^{*}_ug_{d(\mathfrak{s})}^*m_\lambda H_{2l+1, n}g_{d(\mathfrak{t_1})}g_{v_1} + \check{Br}^\lambda_n
\subseteq g^{*}_ug_{d(\mathfrak{s})}^*m_\lambda H_{2l+1, n}g_{v_1} + \check{Br}^\lambda_n.
\end{align*}
The last formula implies that $x^\lambda_{(\mathfrak{s}, u)(\mathfrak{t}, v)}\cdot b'$
can be rewritten as an $R$-linear combination
\begin{align*}
x^\lambda_{(\mathfrak{s}, u)(\mathfrak{t}, v)}\cdot b' &= g^{*}_ug_{d(\mathfrak{s})}^*m_\lambda \big{(} \sum_{  \pi_1 \in S_{2l+1, n}} a_{ \pi_1} g_{ \pi_1} \big{)} g_{v_1} + \check{Br}^\lambda_n\\
&\overset{ \eqref{ct4}}= g^{*}_ug_{d(\mathfrak{s})}^*c_\lambda \big{(} \sum_{  \pi_1 \in S_{2l+1, n}} a_{ \pi_1} e_{(k)} g_{ \pi_1}g_{v_1} \big{)} + \check{Br}^\lambda_n\\
&\overset{L\ref{lm4}}= g^{*}_ug_{d(\mathfrak{s})}^*c_\lambda \big{(} \sum_{  \pi_1 \in S_{2l+1, n}}a_{ \pi_1} ( \sum_{ \substack{ \omega_{\pi_1} \in S_{2k+1, n}\\ v_{\pi_1}\in B_{k,n}}}  a_{( \omega_{\pi_1}, v_{\pi_1})} e_{(k)} g_{ \omega_{\pi_1}} g_{v_{\pi_1}}) \big{)}  + \check{Br}^\lambda_n\\
&=  \sum_{  \pi_1 \in S_{2l+1, n}}a_{ \pi_1}( \sum_{ \substack{ \omega_{\pi_1} \in S_{2k+1, n}\\ v_{\pi_1}\in B_{k,n}}}  a_{( \omega_{\pi_1}, v_{\pi_1})} g^{*}_ug_{d(\mathfrak{s})}^* e_{(k)} (c_\lambda g_{ \omega_{\pi_1}}) g_{v_{\pi_1}})  + \check{Br}^\lambda_n\\
&\overset{ \eqref{ct6}}= \sum_{  \pi_1 \in S_{2l+1, n}} a_{ \pi_1} \big{(} \sum_{ \substack{ \omega_{\pi_1} \in S_{2k+1, n}\\ v_{\pi_1}\in B_{k,n}}}  a_{( \omega_{\pi_1}, v_{\pi_1})} g^{*}_ug_{d(\mathfrak{s})}^*e_{(k)} (\sum_{ \mathfrak{t}_{\omega_{\pi_1}} \in Std(\lambda)} a_{\mathfrak{t}_{\omega_{\pi_1}}}c_{\mathfrak{\boldsymbol{1}}\mathfrak{t}_{\omega_{\pi_1}}} + \check{\mathscr{H}}^\lambda_{2k+1, n}) g_{v_{\pi_1}} \big{)} + \check{Br}^\lambda_n\\
&\overset{}= \sum_{  \pi_1 \in S_{2l+1, n}} a_{ \pi_1} \big{(} \sum_{ \substack{ \omega_{\pi_1} \in S_{2k+1, n}\\ v_{\pi_1}\in B_{k,n}}}  a_{( \omega_{\pi_1}, v_{\pi_1})}  (\sum_{ \mathfrak{t}_{\omega_{\pi_1}} \in Std(\lambda)} a_{\mathfrak{t}_{\omega_{\pi_1}}}g^{*}_ug_{d(\mathfrak{s})}^*e_{(k)}c_{\mathfrak{\boldsymbol{1}}\mathfrak{t}_{\omega_{\pi_1}}}g_{v_{\pi_1}}) \big{)}\\
&+ \sum_{  \pi_1 \in S_{2l+1, n}} a_{ \pi_1} \big{(} \sum_{ \substack{ \omega_{\pi_1} \in S_{2k+1, n}\\ v_{\pi_1}\in B_{k,n}}}  a_{( \omega_{\pi_1}, v_{\pi_1})} g^{*}_ug_{d(\mathfrak{s})}^*e_{(k)} \check{\mathscr{H}}^\lambda_{2k+1, n} g_{v_{\pi_1}} \big{)}  + \check{Br}^\lambda_n\\
\end{align*}
\begin{align*}
&= \sum_{  \pi_1 \in S_{2l+1, n}} a_{ \pi_1} \big{(} \sum_{ \substack{ \omega_{\pi_1} \in S_{2k+1, n}\\ v_{\pi_1}\in B_{k,n}}}  a_{( \omega_{\pi_1}, v_{\pi_1})} (\sum_{ \mathfrak{t}_{\omega_{\pi_1}} \in Std(\lambda)} a_{\mathfrak{t}_{\omega_{\pi_1}}} x^\lambda_{(\mathfrak{s}, u)(\mathfrak{t}_{\omega_{\pi_1}}, v_{\pi_1})})\big{)}\\
&+ \sum_{  \pi_1 \in S_{2l+1, n}} a_{ \pi_1} \big{(} \sum_{ \substack{ \omega_{\pi_1} \in S_{2k+1, n}\\ v_{\pi_1}\in B_{k,n}}}  a_{( \omega_{\pi_1}, v_{\pi_1})} g^{*}_ug_{d(\mathfrak{s})}^*e_{(k)} \check{\mathscr{H}}^\lambda_{2k+1, n} g_{v_{\pi_1}} \big{)}  + \check{Br}^\lambda_n,
\end{align*}
where $a_{ \pi_1}$, $a_{( \omega_{\pi_1}, v_{\pi_1})}$, and $a_{\mathfrak{t}_{\omega_{\pi_1}}}$ are in $R$. By the definition of $\check{Br}^\lambda_n$ in \eqref{basisofBr}, it is obviously that the middle term in the last formula is in $\check{Br}^\lambda_n$. So, the last formula can be rearranged such that 
$$x^\lambda_{(\mathfrak{s}, u)(\mathfrak{t}, v)}\cdot b \equiv
\sum_{(\mathfrak{t}_2,v_2)\in\mathcal{I}_{n}(k,\lambda)}a_{(\mathfrak{t}_2,v_2)}
x^\lambda_{(\mathfrak{s}, u)(\mathfrak{t}_2, v_2)}\ \ \mod
\check{Br}^\lambda_n,$$ 
where $\mathfrak{t}_2:= \mathfrak{t}_{\omega_{\pi_1}} \in Std(\lambda)$, $v_2:= v_{\pi_1} \in B_{k,n}$, and $a_{(\mathfrak{t}_2,v_2)}$ is the corresponding coefficient of $x^\lambda_{(\mathfrak{s}, u)(\mathfrak{t}_2, v_2)}$.
Thus, we get the precise statement~\eqref{ct10}.
\end{proof}
As a consequence of the above theorem,  $\check{Br}_n^\lambda$ is the $R$-module freely generated by the collection \eqref{basisofBr}.

The new basis \eqref{ct5} of the $q$-Brauer algebra can be verified to be a cellular basis in the sense of Graham and Lehrer \cite{GL} by checking conditions of the definition of cellular algebra (see Definition 1.1 in \cite{GL}) as follows:

The $q$-Brauer algebra $Br_n(r^2, q^2)$ has the cell datum $(\Lambda_n, \mathcal{I}_{n}, C, *)$ where
\begin{enumerate}
\item[(C1)] \ \ $\Lambda_n$ is a partially ordered set with the dominance order defined in Section \ref{com}. For each $(k, \lambda) \in \Lambda_n$, $\mathcal{I}_{n}(k, \lambda)$ is a finite set satisfying that 
$$C: \coprod_{(k, \lambda)\in \Lambda_n} \mathcal{I}_{n}(k, \lambda) \times \mathcal{I}_{n}(k, \lambda) \rightarrow Br_n(r^2, q^2)$$ 
determined by the rule 
$C((\mathfrak{s},u),\ (\mathfrak{t},v)) = x^\lambda_{(\mathfrak{s}, u)(\mathfrak{t}, v)}$ is injective map.
\item[(C2)] This condition follows from Theorem \ref{Mthm}(1).
\item[(C3)] This condition is satisfied by Theorem \ref{Mthm}(2).
\item[(C3')] This condition is obtained by applying * to the equation \eqref{ct10}, we obtain 
\begin{align*}
b^*\cdot x^\lambda_{(\mathfrak{t}, v)(\mathfrak{s}, u)} \equiv
\sum_{(\mathfrak{t}_2,v_2)\in\mathcal{I}_{n}(k,\lambda)}a_{(\mathfrak{t}_2,v_2)}
x^\lambda_{(\mathfrak{t}_2, v_2)(\mathfrak{s}, u)}\ \ \mod
\check{Br}^\lambda_n.
\end{align*}
\end{enumerate}

Now we can apply the representation theory of cellular algebras for the $q$-Brauer algebra.
For $(k, \lambda) \in \Lambda_n$, the {\it{Cell}} module, says $C(k, \lambda)$, of the $q$-Brauer algebra is called {\it{Specht}} module and is defined to be the $R$--module freely generated by
\begin{align}\label{ct11}
\left \{x^\lambda_{(\mathfrak{t}, v)} :=m_\lambda g_{d(\mathfrak{t})}g_v 
+\check{Br}_n^\lambda\,|\,(\mathfrak{t},v)\in \mathcal{I}_{n}(k,\lambda)\right\}
\end{align}
and with the right $Br_n(r^2, q^2)$ action
\begin{align*}
x^\lambda_{(\mathfrak{t}, v)}\cdot b+\check{Br}_n^\lambda
= \sum_{(\mathfrak{t}_1,v_1)\in\mathcal{I}_{n}(k,\lambda)}a_{(\mathfrak{t}_1,v_1)} 
x^\lambda_{(\mathfrak{t}_1, v_1)} +\check{Br}_n^\lambda&&\text{for $b\in
Br_n(r^2,q^2)$,}
\end{align*}
where the coefficients $a_{(\mathfrak{t}_1,v_1)}\in R$, for
$(\mathfrak{t}_1,v_1)$ in $\mathcal{I}_{n}(k,\lambda)$, are determined by
the expression~\eqref{ct10}.


\begin{remark} \label{rm2}
1. The cellular basis of the $q$-Brauer algebra in Theorem \ref{Mthm} is a lift of the Murphy bases of the Hecke algebras $H_{2k+1, n}$ for $0 \le k \le [n/2]$. We do not know if this basis is the Murphy basis in the sense of Murphy \cite{M1, M2}. So far, A family of Jucys-Murphy elements of the $q$-Brauer algebra is unknown.

2. Note that the Specht module $C(k, \lambda)$ of the $q$-Brauer algebra in this paper is a lift of the Specht module $S^\lambda$ of the Hecke algebra of the symmetric group in \eqref{ct7}, due to Mathas \cite{Math}. In the case $q=1$, it recovers the Specht module of the classical Brauer algebra used in \cite{HP}.

3. Let $F$ be a field and
$\hat{r},\hat{q},(\hat{q}-\hat{q}^{-1})$ and $(\hat{r}-\hat{r}^{-1})$ be units in $F$. The
assignments $\varphi:r\mapsto \hat{r}$ and $\varphi:q\mapsto\hat{q}$
determine a homomorphism $R\to F$, giving $F$ an
$R$--module structure.  We refer to the specialization
$Br_n(\hat{r}^2, \hat{q}^2)=Br_n(r^2,q^2)\otimes_R \F$ as a $q$-Brauer algebra
over $F$. If $(k, \lambda) \in \Lambda_n$ then the cell module $C(k, \lambda)\otimes_RF$ for
$Br_n(\hat{r}^2, \hat{q}^2)$ admits a symmetric associative bilinear form
which is related to the generic form~\eqref{ct12} in an obvious
way. Similarly, this holds true for the version $Br_n(N)$.

4. Whenever the context is clear and no confusion can
arise, the abbreviation $Br_n(r^2, q^2)$ will be used for $Br_n(\hat{r}^2, \hat{q}^2)$ and similarly, $C(k, \lambda)$ will be used for the
$Br_n(\hat{r}^2, \hat{q}^2)$--module $C(k, \lambda)\otimes_R\F$.

5. In the case $q=1$ the version of Theorem~\ref{Mthm} for $Br_n(N)$ coincides with Enyang's result to the classical Brauer algebra $D_n(N)$ (see~\cite{En}, Theorem~9.1). It implies that  over a field $F$ of any characteristic the other results for the $q$-Brauer algebra $Br_n(N)$ in this article recover those of the classical Brauer algebra.
\end{remark}
The example below illustrates a basis for Specht module.

\begin{example}\label{ex2}
Let $n=5$, $k=1$, and $\lambda=(2,1)$. If $j, i_j$ are integers with
$1\le i_j \le j\le n-1$, write  $t_{j}=1 \text{ or } t_{j}= s_js_{j-1} \cdots  s_{i_j}$, so that\\
$B_{2, 5} = \{ v = t_2t_4 | \ t_j = 1 \text{ or } t_j = s_{j,i_j}, \ 1\le i_j \le j \text{ for } j \in\{1, 2,4\} \}$;
\begin{align*}
B_{1, 5} = \{v = t_2t_3t_4 | \ t_j &= 1 \text{ or } t_j = s_{j,i_j}, \ 1\le i_j \le j \le 4 \} \\
&= \{ \boldsymbol{1},\ s_2,\ s_{2,3},\ s_{2,1},\ s_{2,1}s_3,\ s_{2,1}s_{3,2},\ s_{2,4},\ s_{2,1}s_{3,4},\ s_{2,1}s_{3,2}s_4,\ s_{2,1}s_{3,2}s_{4, 3} \}.
\end{align*}
Since the set of partitions $\{ \mu \ | \mu \rhd \lambda \} = \{ \mu_1 = (3),\ \mu_2 = (1)\}$ we obtain as follows:\\
With $\mu_1 = (3)$ the Young subgroup  $S_{\mu_1} = \{\boldsymbol{1}, s_3, s_4, s_3s_4, s_4s_3, s_4s_3s_4 \}$ and the set of all standard tableau $Std(\mu_1)= \{ \ \mathfrak{t}^{\mu_1}=\text{\tiny\Yvcentermath1$\young(345)$}\ \}$. Hence
 $$m_{\mu_1} = e(\boldsymbol{1}+ g_3+ g_4+ g_3g_4+ g_4g_3+ g_3g_4g_3)
 =e(\boldsymbol{1}+ g_3)(\boldsymbol{1}+ g_4 + g_4g_3).$$
 With $\mu_2 = (1)$ the Young subgroup  $S_{\mu_2} = \{ \boldsymbol{1} \}$,
 $Std(\mu_2)= \{ \ \mathfrak{t}^{\mu_2}=\text{\tiny\Yvcentermath1$\young(5)$}\ \}$ and
 $m_{\mu_2} = e_{(2)}$. Now, by Equation \eqref{bsofBr} the two-sided ideal $\check{Br}^{(2,1)}_5$ has a basis:
 \begin{align} 
\bigg\{
\begin{matrix}
\text{$x^{\mu_1}_{(\mathfrak{s}_1, u_1)(\mathfrak{t}_1, v_1)}$,}\\
\text{$x^{\mu_2}_{(\mathfrak{s}_2, u_2)(\mathfrak{t}_2, v_2)}$}
\end{matrix}
\bigg|\ 
\begin{matrix}
\text{$(\mathfrak{t}_1,v_1),(\mathfrak{s}_1,u_1)\in\mathcal{I}_{n}(l, \mu_1)$,}\\
\text{$ (\mathfrak{t}_2,v_2),(\mathfrak{s}_2,u_2)\in\mathcal{I}_{n}(l, \mu_2)$}
\end{matrix}
\bigg\} =
\bigg\{
\begin{matrix}
\text{$x^{\mu_1}_{(\boldsymbol{1}, u_1)(\boldsymbol{1}, v_1)}$,}\\
\text{$x^{\mu_2}_{(\boldsymbol{1}, u_2)(\boldsymbol{1}, v_2)}$}
\end{matrix}
\bigg|\ 
\begin{matrix}
\text{$v_1, u_1\in B_{1, 5}$}\\
\text{$ v_2, u_2\in B_{2, 5}$}
\end{matrix}
\bigg\}.
\end{align}
In the other hand, we get
$\STD(\lambda)=\left\{\mathfrak{t}^\lambda=\text{\tiny\Yvcentermath1$\young(34,5)$}\,\,,
\mathfrak{t}^\lambda s_4=\text{\tiny\Yvcentermath1$\young(35,4)$}\,\,\right\}$, $S_{\lambda} = \{ \boldsymbol{1}, s_3\}$
and $m_\lambda=e(1+g_3)$.  The basis of the Specht module
$C(1, \lambda)$, of the form displayed in~\eqref{ct11}, is
\begin{align*}
\big\{\ e(1+g_3)g_v +\check{Br}^{(2,1)}_5,\ e(1+g_3)g_4g_v +\check{Br}^{(2,1)}_5 \ | v\in B_{1, 5}\ \big\}.
\end{align*}
\end{example}
As in Proposition 2.4 of~\cite{GL}, the Specht module
$C(k, \lambda)$ for $Br_n(r^2,q^2)$ admits an associative bilinear
form $\langle\,\,,\,\rangle_\lambda:C(k, \lambda) \times C(k, \lambda)\to R$
defined by
\begin{align}\label{ct12}
\langle x^{\lambda}_{(\mathfrak{t}, v)},x^{\lambda}_{(\mathfrak{s}, u)}  \rangle_\lambda m_\lambda\equiv x^{\lambda}_{(\mathfrak{t}, v)} (x^{\lambda}_{(\mathfrak{s}, u)})^* \mod \check{Br}^\lambda_n.
\end{align}
This means 
\begin{align*}
\langle m_{\lambda} g_{d(\mathfrak{t})}g_v+\check{Br}_n^\lambda, m_{\lambda} g_{d(\mathfrak{s})}g_u +\check{Br}_n^\lambda \rangle_\lambda m_\lambda 
&\equiv m_{\lambda} g_{d(\mathfrak{t})}g_v g^*_u g^*_{d(\mathfrak{t})}m_{\lambda}   \mod \check{Br}^\lambda_n.
\end{align*}

\begin{example}\label{bilinear:ex}
Let $n=3$, $k=1$ and $\lambda=(1)$. So that $\check{Br}_3^{(1)}=(0)$ and
$m_{(1)}=e$. We order the basis~\eqref{ct11} for the module
$C(1, (1))$ as $\mathbf{v}_1=e$, $\mathbf{v}_2=eg_2$ and
$\mathbf{v}_3=eg_2g_1$ and, with respect to this ordered basis,
the Gram matrix $\langle\mathbf{v}_i,\mathbf{v}_j \rangle_\lambda$ of the
bilinear form~\eqref{ct12} is (see \cite{N1} p.69 for a detail calculation)
\begin{align*}
\begin{bmatrix}
& a &rq &rq^{3} \\
& rq &q^{2}a+(q^{2}-1)rq &rq^{5} \\
& rq^{3} &rq^{5} &q^{4}a +(q^{4}-1)rq^{3}\\
\end{bmatrix}
,\text{ where \ $a=\dfrac{r-r^{-1}}{q-q^{-1}}$}.
\end{align*}
The determinant of the Gram matrix given above is 
\begin{align}\label{detGr3}
\dfrac {3q^{5}(r^{2} - q^{2})^{2}(q^{4}r^{2} -1)} {r^{3}(q^{2}-1)^{3}}
\end{align}
\end{example}

\section{Representation theory over a field}
Using the basis of Specht modules $C(l, \lambda)$ we have defined the new $F$-bilinear form, $\langle\ ,\ \rangle_\lambda$, for the $q$-Brauer algebra. This bilinear form differs from the one given in Definition 4.21\cite{N}. In detail, instead of determining the bilinear form via the known bilinear forms, including two bilinear forms of both the Hecke algebra and the iterated inflation's algebra, our bilinear form is directly defined using \eqref{ct11}. This enable us to give a classification of simple $Br_n(r^2, q^2)$-modules only by using explicit calculations in Theorem \ref{irremodule}.  Also notice that in a restriction to subalgebra $H_n$ the new bilinear form recovers the known one for the Hecke algebra.   

Using the general theory of cellular algebras we obtain some results about the \break representation theory of $q$-Brauer algebras. 
From now on, let $F$ be an arbitrary field of characteristic $p \geq 0$. Denote
$$rad(C(k, \lambda))=\{ x\in C(k, \lambda) |\ \langle x, y \rangle_\lambda=0 \text{ for all } y\in C(k, \lambda) \}$$ 
and $$D(k, \lambda)=C(k, \lambda)/rad(C(k, \lambda)).$$ 

The following are special cases of results in \cite{GL}.

{\bf{Statement 1.}} For $(k, \lambda) \in \Lambda_n$, $r,\ q$ and $(r - r^{-1})/(q-q^{-1})$ invertible elements in an arbitrary field $F$, let $Br_n(r^{2}, q^{2})$ be a $q$-Brauer algebra over $F$.  Then 
\begin{enumerate}
\item $rad(C(k, \lambda))$ is a $Br_n(r^{2}, q^{2})$-submodule of $C(k, \lambda)$;
\item If $D(k, \lambda) \neq 0$ then
\begin{enumerate}
	\item $D(k, \lambda)$ is simple;
	\item $rad(C(k, \lambda))$ is the radical of the $Br_n(r^{2}, q^{2})$-module $C(k, \lambda).$
\end{enumerate}
\end{enumerate}

{\bf{Statement 2.}} For $(k, \lambda), (l, \mu) \in \Lambda_n$, let $Br_n(r^{2}, q^{2})$ be a $q$-Brauer algebra over an \break arbitrary field $F$. Suppose $M$ is a $Br_n(r^{2}, q^{2})$-submodule of $C(k, \lambda)$ and\\ 
\begin{center}
$\varphi:~C(l, \mu)~\longrightarrow~C(k, \lambda)/M$
\end{center}
is a $Br_n(r^{2}, q^{2})$-module homomorphism, and $\langle ,  \rangle_\mu \neq 0$. Then
\begin{enumerate}
\item $\varphi \neq 0$ only if $\lambda \unrhd \mu$.
\item If $\lambda = \mu$, then there are elements $0 \neq r_0,\ r_1\in F$ such that for all $x \in C(l, \mu)$, we have $r_0\varphi(x) = r_1x +M$.
\end{enumerate}

For $(k, \lambda), (l, \mu) \in \Lambda_n$ and $D(l, \mu)\neq 0$, let $d_{\lambda \mu} = [C(k, \lambda)\ :\ D(l, \mu)]$ be the composition multiplicity of $D(l, \mu)$ in $C(k, \lambda)$.

The next statement provides a classification of the simple $Br_n(r^{2}, q^{2})$--modules. This result is an analogue of that for the Hecke algebra due to Dipper and James (see \cite{DJ1}, Theorem~7.6). 

\begin{theorem} \label{irremodule} For $(k, \lambda) \in \Lambda_n$, $r,\ q$ and $(r - r^{-1})/(q-q^{-1})$ invertible elements in an arbitrary field $F$, let $Br_n(r^{2}, q^{2})$ be a $q$-Brauer algebra over $F$. Then
\begin{enumerate}
\item The set
$\{ D(l, \mu) |\  (l, \mu)\in \Lambda_n \text{ and $\mu$ is an $e(q^{2})$-restricted partition} \}$
is a complete set of pairwise non-isomorphic simple $Br_n(r^{2}, q^{2})$-modules.
\item For $(k, \lambda), (l, \mu) \in \Lambda_n$, suppose that $\mu$ is an $e(q^2)$-restricted partition. Then $d_{\mu \mu}~=~1$ and $d_{\lambda \mu}  \neq 0$ only if $\lambda \unrhd\mu$.
\end{enumerate}
\end{theorem}

\begin{proof}
(1). Since the $q$-Brauer algebra is cellular, it follows from Theorem~3.4~\cite{GL} that the set 
\begin{center} $\{ D(l, \mu) |\  D(l, \mu) \neq 0  \text{ for partition $\mu$ of $n-2l$,\ $0 \le l \le [n/2]$} \}$ \end{center}
is a complete set of pairwise non-isomorphic simple $Br_n(r^{2}, q^{2})$-modules. The remainder of proof is to show that $D(l, \mu) \neq 0$ if and only if $\mu$ is an $e(q^{2})$-restricted partition of $n-2l$. 

Indeed, pick up two non-zero elements $x^\mu_{(\mathfrak{s}, u)} = m_\mu g_{d(\mathfrak{s})}g_u+\check{Br}_n^\mu \text{ and } x^\mu_{(\mathfrak{t}, v)} = m_\mu g_{\mathfrak{t}}g_v+~\check{Br}_n^\mu$ in $C(l, \mu)$ with arbitrary pairs $(\mathfrak{s},\ u),\ (\mathfrak{t}, v) \in \mathcal{I}_{n}(l,\mu)$. This yields, using \eqref{ct12},

\begin{align} \label{ct13}
\langle x^\mu_{(\mathfrak{s}, u)}, x^\mu_{(\mathfrak{t}, v)}  \rangle_\mu m_\mu &= \langle m_\mu g_{d(\mathfrak{s})}g_u +\check{Br}_n^\mu,\ m_\mu g_{d(\mathfrak{t})} g_v +\check{Br}_n^\mu \rangle_\mu m_\mu\\
&\overset{\eqref{ct12}}{\equiv} m_{\mu} g_{d(\mathfrak{s})}g_v g^*_u g^*_{d(\mathfrak{t})}m_{\mu} \mod \check{Br}^\mu_n \notag \\
&\overset{\eqref{ct4}}{\equiv} e_{(l)} (c_\mu g_{d(\mathfrak{s})}) g_v g^*_u ( g^*_{d(\mathfrak{t})}c_{\mu}) e_{(l)} \mod \check{Br}^\mu_n\notag \\
&\overset{\eqref{ct7'}}{\equiv} e_{(l)} (c_{\mathfrak{\boldsymbol{1}}\mathfrak{s}} g_v) (g^*_u c_{\mathfrak{t}\mathfrak{\boldsymbol{1}}}) e_{(l)} \mod \check{Br}^\mu_n \notag \\
&\overset{\eqref{ct6}}{\equiv} e_{(l)} (\sum_{\mathfrak{s}_1 \in Std(\mu)} a_{\mathfrak{s}_1} c_{\mathfrak{\boldsymbol{1}}\mathfrak{s}_1}+ \check{\mathscr{H}}^\mu_{2l+1, n})(\sum_{\mathfrak{t}_1 \in Std(\mu)} a_{\mathfrak{t}_1}c_{\mathfrak{t}_1 \mathfrak{\boldsymbol{1}}}+ \check{\mathscr{H}}^\mu_{2l+1, n}) e_{(l)} \mod \check{Br}^\mu_n \notag \\
&\overset{\eqref{ct7}}{\equiv} e_{(l)} (\sum_{\mathfrak{s}_1 \in Std(\mu)} a_{\mathfrak{s}_1} c_{\mathfrak{s}_1})(\sum_{\mathfrak{t}_1 \in Std(\mu)} a_{\mathfrak{t}_1}c^*_{\mathfrak{t}_1}) e_{(l)} \mod \check{Br}^\mu_n \notag \\
&\overset{}{\equiv} e_{(l)} \sum_{\mathfrak{s}_1, \mathfrak{t}_1 \in Std(\mu)} a_{\mathfrak{s}_1}a_{\mathfrak{t}_1} (c_{\mathfrak{s}_1} c^*_{\mathfrak{t}_1}) e_{(l)} \mod \check{Br}^\mu_n\notag \\
&\overset{\eqref{bilinearform1}}{\equiv} e_{(l)} \sum_{\mathfrak{s}_1, \mathfrak{t}_1 \in Std(\mu)} a_{\mathfrak{s}_1} a_{\mathfrak{t}_1} (\langle  c_{\mathfrak{s}_1},\ c_{\mathfrak{t}_1} \rangle  c_{\mu} + \check{\mathscr{H}}^\mu_{2l+1, n}) e_{(l)} \mod \check{Br}^\mu_n \notag 
\end{align}
\begin{align*}
&\overset{\eqref{ct4}, L\ref{tc1}(2)\ \text{and}\ \eqref{bsofBr}}{\equiv} \sum_{\mathfrak{s}_1, \mathfrak{t}_1 \in Std(\mu)} \big{(}\dfrac{r-r^{-1}}{q-q^{-1}}\big{)}^l a_{\mathfrak{s}_1} a_{\mathfrak{t}_1} \langle  c_{\mathfrak{s}_1},\ c_{\mathfrak{t}_1} \rangle m_{\mu} \mod \check{Br}^\mu_n \notag 
 \end{align*}
 where $a_{\mathfrak{s}_1}, \ a_{\mathfrak{t}_1}$ are coefficients in $F$.
 
 Now, if $\mu$ is an $e(q^2)$-restricted partition of $n-2l$ then by Theorem~\ref{simplemd1}(1) it implies $D^\mu \neq 0$, that is, there exist $\mathfrak{s}_0,\ \mathfrak{t}_0 \in Std(\mu)$ such that 
 $\langle c_{\mathfrak{s}_0},\ c_{\mathfrak{t}_0} \rangle \neq 0$. Subsequently, fix two basis elements $x^\mu_{(\mathfrak{s}_0, \boldsymbol{1})} = m_\mu g_{d(\mathfrak{s}_0)}+\check{Br}_n^\mu  \text{ and } x^\mu_{(\mathfrak{t}_0, \boldsymbol{1})} = m_\mu g_{d(\mathfrak{t}_0)}+\check{Br}_n^\mu$ in $C(k, \lambda)$. As a special case of calculation \eqref{ct13} we obtain 
\begin{align*}
\langle x^\mu_{(\mathfrak{s}_0, \boldsymbol{1})} , x^\mu_{(\mathfrak{t}_0, \boldsymbol{1})}   \rangle_\mu m_\mu \equiv \big{(}\dfrac{r-r^{-1}}{q-q^{-1}} \big{)}^l \langle c_{\mathfrak{s}_0}, c_{\mathfrak{t}_0} \rangle m_\mu \ \mod\check{Br}^\mu_n
 \end{align*}
 and hence $\langle x^\mu_{(\mathfrak{s}_0, \boldsymbol{1})} , x^\mu_{(\mathfrak{t}_0, \boldsymbol{1})}  \rangle_\mu =\big{(}\dfrac{r-r^{-1}}{q-q^{-1}}\big{)}^l \langle c_{\mathfrak{s}_0}, c_{\mathfrak{t}_0} \rangle \neq 0$. This result implies that $D(l, \mu) \neq 0$. 
 
 Conversely, if $\mu$ is not $e(q^2)$-restricted then by Theorem~\ref{simplemd1}(1), $D^\mu =0$. This means that $ \langle c_{\mathfrak{s}}, c_{\mathfrak{t}} \rangle = 0$ for any $c_{\mathfrak{s}}, c_{\mathfrak{t}}\in S^{\mu}$. Applying calculation \eqref{ct13} it implies 
$\langle x^\mu_{(\mathfrak{s}, u)}, x^\mu_{(\mathfrak{t}, v)}  \rangle_\mu = 0$ for all $x^\mu_{(\mathfrak{s}, u)}, x^\mu_{(\mathfrak{t}, v)} \in C(l, \mu)$, namely, $D(l, \mu) = 0$.

(2). This statement follows by applying the general theory of cellular algebras and Proposition~3.6~\cite{GL}.
\end{proof}

\begin{corollary} \label{cor3}
For $(k, \lambda) \in \Lambda_n$, $r,\ q$ and $(r - r^{-1})/(q-q^{-1})$ invertible elements in an arbitrary field $F$, let $Br_n(r^{2}, q^{2})$ be a $q$-Brauer algebra over $F$. The following statements are equivalent.
\begin{enumerate}
\item $Br_n(r^{2}, q^{2})$ is semisimple;
\item $C(k, \lambda) = D(k, \lambda)$ for all $(k, \lambda) \in \Lambda_n$; and,
\item The $F$-bilinear form $\langle\ ,\  \rangle_\lambda$ (cf. \eqref{ct12}) is non-degenerate for all $(k, \lambda) \in \Lambda_n$.
\end{enumerate}
\end{corollary}

\begin{remark} The same results as Theorem \ref{irremodule} and Corollary \ref{cor3} hold true for the version $Br_n(N)$ of the $q$-Brauer algebra. Furthermore, when $q=1$ then the statement in Theorem~\ref{irremodule} recovers that for the classical Brauer algebra with non-zero parameter which was shown in Theorem 4.17\cite{GL} by Graham and Lehrer. Also notice that in this case the cell module of the Brauer algebra in \cite{GL} is dual to the one in this paper.
\end{remark}

\section{Is the $q$-Brauer algebra generically isomorphic with the BMW-algebra?}
In this section, we answer the question whether the $q$-Brauer algebra is isomorphic with the BMW-algebra? Combining the cellularity of the $q$-Brauer algebra, explicit calculations on basis \eqref{ct11} of the Specht modules $C(k, \lambda)$ and concrete examples, we show that in general the answer is {\bf{"No"}}. To this end, we need the following results.

\smallskip
\begin{proposition} \label{criterion1}
Let $Br_n(r^{2}, q^{2})$ be the $q$-Brauer algebra over an arbitrary field $F$ with invertible elements $r,\ q \text{ and } \dfrac{r-r^{-1}}{q-q^{-1}} \in F$. Then
\begin{enumerate}
\item $Br_2(r^{2}, q^{2})$ is semisimple if and only if $e(q^{2}) > 2$.
\item $Br_3(r^{2}, q^{2})$ is semisimple if and only if $e(q^{2}) > 3$ and $\dfrac{3q^{5}(r^{2} - q^{2})^{2}(q^{4}r^{2} -1)} {r^{3}(q^{2}-1)^{3}} \neq 0$.
\end{enumerate}
\end{proposition}
\begin{proof}
If $n=2$ and $\lambda$ is a partition of $2$, then the cell modules $C(0, \lambda)$ coincide with the cell modules $S^{\lambda}$ of the Hecke algebra $H_2$ and this yields $\langle\,\,,\,\rangle \equiv \langle\,\,,\,\rangle_\lambda$.
By Corollary~\ref{cor2}, the $F$-bilinear form $\langle\,\,,\,\rangle$ on $H_2$ is non-degenerate if and only if $e(q^2) > 2$.
If $\lambda = \varnothing$, then $\check{Br}_2^{\lambda} \equiv \check{\mathscr{H}}^\lambda_2 = F$ and $m_{\lambda}=e$.
As shown in \eqref{ct11}, the cell module $C(1, \lambda)$ has basis
$$\{ \ eg_v + \check{Br}_2^{\lambda}\ | \ v\in B_{1, 2} = \{ \boldsymbol{1}\}\ \} = \{ e \}.$$
The Gram determinant with respect to this basis is $\langle e\,,\ e  \rangle_\lambda e =e^2 \overset{(E_1)}= \dfrac{r-r^{-1}}{q-q^{-1}}e$, that is,
$\langle e\,,\ e  \rangle_\lambda = \dfrac{r-r^{-1}}{q-q^{-1}} \neq 0.$
Now, the first statement follows from Corollary \ref{cor3}.

In the case $n=3$ and $\lambda$ is a partition of $3$, then using the same argument as above yields that the cell modules $C(0, \lambda)$ coincide with the cell modules $S^{\lambda}$ of the Hecke algebra $H_3$, and this implies $\langle\,\,,\,\rangle \equiv \langle\,\,,\,\rangle_\lambda$. By Corollary \ref{cor2} the $F$-bilinear form $\langle\,\,,\, \rangle_\lambda$ is non-degenerate if and only if $e(q^2) > 3$.
Otherwise, if $n=3$ and $\lambda = (1)$ then applying Example~\eqref{bilinear:ex}, the Gram determinant on $C(1, \lambda)$ is non-zero if and only if 
$$\dfrac{3q^{5}(r^{2} - q^{2})^{2}(q^{4}r^{2} -1)} {r^{3}(q^{2}-1)^{3}} \neq 0.$$
Hence, we get the statement (2) by using Corollary~\ref{cor3}.
\end{proof}
If replacing the version $Br_n(r^{2}, q^{2})$ by $Br_n(N)$ or $Br_n(r, q)$ used by Wenzl \cite{W3} and  Dung \cite{N}, then the results are the following.

\begin{proposition} \label{criterion2}
Let $Br_n(r, q)$ be the $q$-Brauer algebra over an arbitrary field $F$ with invertible elements $r,\ q \text{ and } \dfrac{r-1}{q-1}$ in $F$. Then
\begin{enumerate}
\item $Br_2(r, q)$ is semisimple if and only if $e(q) > 2$.
\item $Br_3(r, q)$ is semisimple if and only if $e(q) > 3$ and $\dfrac{3q(r - q)^{2}(q^{2}r -1)} {(q - 1)^{3}} \neq 0$.
\end{enumerate}
\end{proposition}
The proof is similar to the one above, using Section 3 in~\cite{N} for detail calculations. 

\begin{proposition} \label{criterion3} Let $N \in \Z\setminus \{0\}$ and 
$Br_n(N)$ be the $q$-Brauer algebra over an arbitrary field $F$ with $0 \neq q, [N] \in F$. Then
\begin{enumerate}
\item $Br_2(N)$ is semisimple if and only if $e(q) > 2$.
\item $Br_3(N)$ is semisimple if and only if $e(q) > 3$ and 
\begin{center}
 $3q^4(q^N - q[N])([N] + q^{N+1} + q^{N+3}) \neq 0$. \end{center}
\end{enumerate}
\end{proposition}
The proof uses the same arguments as in Proposition \ref{criterion1}, applying definition of $Br_n(N)$ given in Remark \ref{rm1}(2) for calculations. 

\begin{remark}
1. Notice that if $q=1$ then $e(q^2)$ (resp. $e(q)$) is equal to the characteristic $p$ of the field $F$. It implies that for $r=q^N$ with $N \in \Z \setminus \{0\}$ and the limit $q \rightarrow 1$, our results above recover these ones for the classical Brauer algebra $D_n(N)$ due to Rui~\cite{R} in the case $n\in \{2, 3\}$.
In particular, when $Lim_{q \rightarrow 1}\dfrac{r-r^{-1}}{q-q^{-1}}= N$ and
\begin{align*}
\hspace{1cm} &Lim_{q \rightarrow 1}\dfrac {3q^{5}(r^{2} - q^{2})^{2}(q^{4}r^{2} -1)} {r^{3}(q^{2}-1)^{3}} =
Lim_{q \rightarrow 1}\dfrac{3q^{5}(q^{2N} - q^{2})^{2}(q^{4}q^{2N} -1)} {q^{3N}(q^{2}-1)^{3}} \\
&= Lim_{q \rightarrow 1}\dfrac{3q^{9}} {q^{3N}}\cdot \dfrac{(q^{2(N-1)} - 1)^{2}}{(q^{2}-1)^{2}} \cdot \dfrac{(q^{2(N+2)} -1)}{(q^{2}-1)}
= 3(N - 1)(N +2),
\end{align*}
then $Br_n(q^{2N}, q^2) \equiv D_n(N)$ over the field $F$ in which the limit $q \rightarrow 1$ can be formed. Applying Proposition \ref{criterion1} it implies the following: Over the complex field $Br_2(q^{2N}, q^2)$ is semisimple if and only if $N \neq 0$ and $Br_3(q^{2N}, q^2)$ is semisimple if and only if $N~\not\in~\{-2, 0, 1\}$; over arbitrary field of characteristic $p > 0$, $Br_2(q^{2N}, q^2)$ is semisimple if and only if $N \neq 0$ and $p>2$, and $Br_3(q^{2N}, q^2)$ is semisimple if and only if $N \not\in \{-2, 0, 1\}$ and $p>3$.  These imply that in the limit $q \rightarrow 1$ Proposition \ref{criterion1} recovers Theorems~1.2(a) and 1.3(a)~in \cite{R} for $n \in \{2, 3\}$. The other computation for the version $Br_n(r, q)$ in Proposition \ref{criterion2} is left to the reader.

Similarly, in the case $q=1$, $Br_n(N)$ coincides with the classical Brauer algebra $D_n(N)$ over arbitrary field $F$, $charF = p \geq 0$. A direct calculation yields that Proposition \ref{criterion3} recovers Theorems~1.2(a) and 1.3(a) for $n \in \{2, 3\}$ in \cite{R}.  

2. Over the field of characteristic zero results above agree with Wenzl's results for $n\in\{2, 3 \}$ (see Theorem~5.3~\cite{W2}). In particular, for $Br_n(r^2, q^2)$ (resp. $Br_n(r, q)$) the pair of parameters $(\xi, \rho)$ in his theorem is replaced by $(q^2,\ r)$ (resp. $(q, r)$), respectively. And for $Br_n(N)$, the pair of parameters $(\xi, \rho)$ is replaced by $(q^2,\ q^N)$. 

3. Propositions~\ref{criterion1} and \ref{criterion2} imply a negative answer for the question about the existence of an isomorphism between the $q$~-~Brauer algebra $Br_n(r^2, q^2)$ (resp. $Br_n(r, q)$)  and the BMW- algebra $\mathscr{B}_n$. Three following examples illustrate Claim \ref{claim}.

{\it{In the two following examples, with a same parameter value the BMW-algebra is not simple, but the $q$-Brauer algebra is semisimple}}

\begin{example}\label{ex10}
We consider both algebras $Br_{3}(r^2, q^{2})$ and $\mathscr{B}_{3}$ over the complex field. These algebras simultaneously depend on two parameters $r$ and $q$.
Fixing $r= q^{-1}$ and $q^{2} = -i$, then by Theorem~5.9(b)~\cite{RS} the BMW-algebra $\mathscr{B}_{3}$ is not semisimple
since \begin{center} $q^{4}+1 = (-i)^{2} + 1 = 0.$ \end{center}
On the other hand, both $[m]_{q^{2}} = 1 + q^{2} = 1 - i \neq 0$ and
 $$[m]_{q^{2}} = 1 + q^{2} + (q^{2})^2 = 1 - i + (-i)^{2} = -i \neq 0 \text{, namely, } e(q^{2})=m > 3.$$
Moreover, for $r= q^{-1}$ and $q^{2} = -i$ a direct calculation yields
$$\dfrac {3q^{5}(r^{2} - q^{2})^{2}(q^{4}r^{2} -1)} {r^{3}(q^{2}-1)^{3}}
= \dfrac {3q^{5}(q^{-2} - q^{2})^{2}(q^{4}q^{-2} -1)} {q^{-3}(q^{2}-1)^{3}}= 6i \neq 0.$$
Therefore, applying Proposition~\ref{criterion1}(2) the $q$-Brauer algebra $Br_{3}(r^2, q^{2})$ is semisimple.
\end{example}

\medskip
\begin{example}\label{ex11} 
With respect to the version $Br_{3}(r, q)$ and $\mathscr{B}_{3}$ over the complex field, we choose $r= q^{-1}$ and $q = i\sqrt{i}$. By Theorem~5.9(b)~\cite{RS} the BMW-algebra $\mathscr{B}_{3}$ is not semisimple
since $q^{4}+1 = (i\sqrt{i})^{4} + 1 = 0.$\\
In other words, both $[m]_{q} = 1 + q = 1 + i\sqrt{i} \neq 0$ and
$$[m]_{q} = 1 + q + q^{2} = 1 + i\sqrt{i} + (i\sqrt{i})^{2} = i\sqrt{i} \neq 0 \text{, namely, } e(q)=m > 3.$$
By a direct calculation, for $r= q^{-1}$ and $q = i\sqrt{i}$ it yields 
$$\dfrac {3q(r - q)^{2}(q^{2}r -1)} {(q-1)^{3}}
= \dfrac {3q(q^{-1} - q)^{2}(q^{2}q^{-1} -1)} {(q-1)^{3}}
= 3q^{-1}= 3 (i \sqrt{i})^{-1} \neq 0.$$
Hence, by Proposition \ref{criterion2}(2) the $q$-Brauer algebra $Br_{3}(r, q)$ is semisimple.

The result is illustrated in the following table: 
\begin{align*}
\begin{tabular}[c]{| c  | c | c |}
\hline
$\C$ &BMW-algebra &$q$-Brauer algebra\\
\hline
$(r, q^{2})= (q^{-1}, -i)$& $\mathscr{B}_{3}$ is not semisimple &$Br_{3}(r^2, q^{2})$ is semisimple\\
\hline
$(r, q)= (q^{-1}, i\sqrt{i})$& $\mathscr{B}_{3}$ is not semisimple & $Br_{3}(r, q)$ is  semisimple\\
\hline
\end{tabular}
\end{align*}
\end{example}

\medskip
{\it{The next example shows that over the field of characteristic $p=5$ the BMW-algebra is not semisimple with total twelve parameter values, but the $q$-Brauer algebra is not semisimiple with less than four parameter values.}}

\begin{example}\label{ex12} Over the prime field $\F_5$ if $q \in \{\bar{2},\ \bar{3} \}$, then it is obvious that\\ $[m]_{q^{2}}=\bar{1} + q^{2} = \bar{0}$ and hence $e(q^{2}) \le 2$. Applying Theorem 5.9~in~\cite{R} the BMW-algebra $\mathscr{B}_{2}$ is not semisimple for all $r \in \F_5 \setminus \{\bar{0}\}$. Otherwise, with $q \in \F_5 \setminus \{\bar{0},\ \bar{2},\ \bar{3} \}$ a direct calculation implies that $e(q^{2}) > 2$, and by Theorem 5.9~\cite{R} $\mathscr{B}_{2}$ is not semisimple for $r \in \{ q^{-1}, -q\} = \{\bar{1},\ \bar{4} \}$. Thus, there totally exist twelve value pairs $(r, q)$ such that the BMW-algebra $\mathscr{B}_{2}$ is not semisimple over the field $\F_5$.

Applying Proposition \ref{criterion1}(1) the $q$-Brauer algebra $Br_2(r^2, q^2)$ over the field $\F_5$ is  not semisimple if and only if $q \in \{\bar{2},\ \bar{3} \}$ and $r \in \F_5 \setminus \{\bar{0}\}$ such that $(r-r^{-1})/(q-q^{-1}) \neq 0$. \hspace{2cm} Direct calculation yields $Br_2(r^2, q^2)$ over the field $\F_5$ is not semisimple for all parameters $q \in \{\bar{2},\ \bar{3} \}$ and $r\in \{\bar{2},\ \bar{3} \}$. This means that there are totally such four value pairs $(r, q)$.  

Similarly, on the version $Br_n(r, q)$ Proposition \ref{criterion2}(1) implies that the $q$-Brauer algebra $Br_2(r, q)$ over the field $\F_5$ is  not semisimple if and only if $q \in \{\bar{4}\}$ and $r \in \F_5 \setminus \{\bar{0}\}$ such that $(r-1)/(q-1) \neq 0$. That is, $Br_2(r, q)$ over the field $\F_5$ is not semisimple for all parameters $q =\bar{4}$ and $r\in \{\bar{2},\ \bar{3},\ \bar{4} \}$. 
 
The total parameter values, such that the algebras are not semisimple, are  summarized in the following table.
\begin{align*}
\begin{tabular}[c]{| c  | c | }
\hline
  The non-semisimple case&$\F_5 \times \F_5 $  \\
\hline
The  BMW-algebra $\mathscr{B}_{2}$ & $(r, q) \in (\{\bar{1}, \bar{2}, \bar{3}, \bar{4} \}\times \{\bar{2}, \bar{3} \}) \cup (\{\bar{2}, \bar{3} \}\times \{\bar{1}, \bar{4}) \}$\\
\hline
The $q$-Brauer algebra $Br_2(r^2, q^2)$ & 
 $(r, q) \in \{\bar{2},\ \bar{3} \}\times \{ \bar{2}, \bar{3} \}$\\
 \hline
 The $q$-Brauer algebra $Br_2(r, q)$ & 
 $(r, q) \in \{\bar{2}, \bar{3}, \bar{4} \}\times \{ \bar{4} \}$\\
 \hline
\end{tabular}
\end{align*}
\end{example}
\end{remark}
Thus, these examples imply that in general there does not exist an algebra isomorphism between the $q$-Brauer algebra and the BMW-algebra.


\bibliographystyle{plain}       






\end{document}